\documentclass{article}
\usepackage[left=30truemm, right=30truemm]{geometry}
\usepackage{amsmath,amssymb,amsthm}
\usepackage{bm}
\usepackage{mathtools}
\usepackage[shortlabels]{enumitem}
\usepackage{float}
\usepackage{cite}

\usepackage{xcolor,graphicx}
\usepackage{tikz}
\usetikzlibrary{calc}
\usepackage{subcaption}

\usepackage[pdfauthor={derajan},pdfstartview=XYZ,bookmarks=true,
colorlinks=true,linkcolor=blue,urlcolor=magenta,citecolor=blue,pdftex,bookmarks=true,linktocpage=true,hyperindex=true]{hyperref}

\theoremstyle{definition}
\newtheorem{theorem}{Theorem}[]

\newtheorem{corollary}[theorem]{Corollary}

\newtheorem{lemma}[theorem]{Lemma}
\newtheorem{claim}[]{Claim}

\newtheorem{conjecture}[theorem]{Conjecture}
\newtheorem{observation}[theorem]{Observation}
\newtheorem{case}[]{Case}
\newtheorem{subcase}{Subcase}[case]

\numberwithin{equation}{section}

\title{Odd coloring of $k$-trees}
\author{
  Masaki Kashima\thanks{Faculty of Science and Technology, Keio University, Yokohama, Japan. email: masaki.kashima10@gmail.com}\qquad
  Kenta Ozeki\thanks{Faculty of Environment and Information Sciences, Yokohama National University, Yokohama, Japan. email: ozeki-kenta-xr@ynu.ac.jp}
}

\begin{document}
\maketitle

\begin{abstract}
  An odd coloring of a graph is a proper coloring such that every non-isolated vertex has a color that appears at an odd number of its neighbors. 
  This notion was introduced by Petr\v{s}evski and \v{S}krekovski in 2022.
  In this paper, we focus on odd coloring of $k$-trees, where a $k$-tree is a graph obtained from the complete graph of order $k+1$ by recursively adding a new vertex that is joined to a clique of order $k$ in the former graph.
  It follows from a result of Cranston, Lafferty, and Song in 2023 that every $k$-tree is odd $(2k+1)$-colorable.
  We improve this bound to show that every $k$-tree is odd $\left(k+2\left\lfloor\log_2 k\right\rfloor+3\right)$-colorable.
  Furthermore, when $k=2,3$, we show the tight bound that every 2-tree is odd $4$-colorable and that every 3-tree is odd $5$-colorable.
\end{abstract}

\section{Introduction}\label{section:intro}

Throughout the paper, we only consider simple, finite, and undirected graphs.
For integers $k$ and $\ell$ with $k\leq \ell$, let $[k,\ell]$ denote the set of integers at least $k$ and at most $\ell$.
In particular, for a positive integer $k$, let $[k]$ denote $\{1,2,\dots , k\}$.

A $k$-coloring of a graph $G$ is a map $\varphi$ from the vertex set $V(G)$ to the set $[k]$ such that $\varphi(u)\neq \varphi(v)$ for every edge $uv$ of $G$.
For a $k$-coloring $\varphi$ of $G$, we say that a vertex $v$ of $G$ \emph{satisfies the odd condition with respect to $\varphi$} if $|\varphi^{-1}(i)\cap N_G(v)|$ is odd for some color $i\in [k]$.
An \emph{odd $k$-coloring} of $G$ is a $k$-coloring $\varphi$ of $G$ such that every non-isolated vertex of $G$ satisfies the odd condition with respect to $\varphi$.
This notion was recently introduced by Petru\v{s}evski and \v{S}krekovski~\cite{PS2022}, and has been studied in many publications \cite{CPS2022, CCKP2023, C2024, CLS2023, DOP2024, KMOOT2024+, KZ2024, MSTY2024, PP2023, PS2022, TY2023, WY2024}.
This paper focuses on odd coloring of $k$-trees.

For a positive integer $k$, a graph $G$ is a \emph{$k$-tree} if there is a sequence $G_0, G_1, \dots , G_s$ of induced subgraphs of $G$ such that 
\begin{itemize}
  \item $G_0$ is isomorphic to $K_{k+1}$, $G_s=G$, and
  \item for every $i\in [s]$, there is a vertex $v_{i+k+1}\in V(G_i)\setminus V(G_{i-1})$ such that $V(G_i)=V(G_{i-1})\cup \{v_{i+k+1}\}$ and $N_G(v_{i+k+1})\cap V(G_{i-1})$ is a clique of order $k$ of $G$.
\end{itemize}
By setting $V(G_0)=\{v_1,v_2,\dots , v_{k+1}\}$, we obtain an ordering $v_1,v_2,\dots , v_n$ of the vertex set of $G$ where $n=s+k+1$.
We call this ordering an \emph{addition ordering} of a $k$-tree $G$.
Note that a sequence of subgraphs with the above properties is not unique, and thus an addition ordering of a given $k$-tree is not unique.
For every addition ordering of a $k$-tree $G$, $N_G(v_i)\cap \{v_j\mid 1\leq j\leq i-1\}$ is a clique of order $k$ for every $i\in [s]$, and in particular, $G$ is $k$-degenerate.
Thus, every $k$-tree is $(k+1)$-colorable, and furthermore, every $k$-tree is uniquely $(k+1)$-colorable.
On the other hand, the bound does not hold for odd coloring of $k$-trees.
Indeed, let $G$ be a $k$-tree such that $V(G)=\{v_i\mid 1\leq i\leq k\}\cup \{u_j\mid 0\leq j\leq k\}$, $\{v_i\mid 1\leq i\leq k\}\cup \{u_0\}$ is a clique of $G$, and $N_G(u_j)=\{v_i\mid 1\leq i\leq k, i\neq j\}\cup \{u_0\}$ for each $j\in [k]$. 
Let $\varphi:V(G)\to [k+1]$ be a proper $(k+1)$-coloring of $G$.
Since every $k$-tree is uniquely $(k+1)$-colorable, without loss of generality, we may assume that $\varphi(v_i)=\varphi(u_i)=i$ for each $i\in [k]$ and $\varphi(u_0)=k+1$.
Then $\varphi^{-1}(i)\cap N_G(u_0)=\{v_i,u_i\}$ for every $i\in [k]$ and thus $u_0$ does not satisfy the odd condition with respect to $\varphi$.
Hence, $G$ is not odd $(k+1)$-colorable.

A natural question is whether the odd chromatic number of a $k$-tree is bounded by a function on $k$.
The answer is yes, according to the following theorem by Cranston, Lafferty, and Song~\cite{CLS2023}.
A class $\mathcal{F}$ of graphs is said to be \emph{minor-closed} if for every graph $G\in \mathcal{F}$ and for every minor $H$ of $G$, $H$ is also contained in $\mathcal{F}$.

\begin{theorem}[\cite{CLS2023}]\label{thm:minor closed}
  Let $\mathcal{F}$ be a minor-closed family of graphs such that every graph in $\mathcal{F}$ is $d$-degenerate.
  Then, every graph in $\mathcal{F}$ is odd $(2d+1)$-colorable.
\end{theorem}

For a positive integer $k$, let $\mathcal{G}_k$ be the set of all graphs with tree-width at most $k$.
It is known that $\mathcal{G}_k$ is a minor-closed family and that every graph in $\mathcal{G}_k$ is $k$-degenerate.
For every positive integer $k$, every $k$-tree has tree-width exactly $k$, and thus the following corollary immediately follows from Theorem~\ref{thm:minor closed}.

\begin{corollary}\label{cor:known bound}
  Let $k$ be a positive integer.
  Every graph of tree-width at most $k$ is odd $(2k+1)$-colorable.
  In particular, every $k$-tree is odd $(2k+1)$-colorable.
\end{corollary}

Our first main result improves the upper bound of the odd chromatic number of $k$-trees as follows.

\begin{theorem}\label{thm:ktree}
  For every positive integer $k$, every $k$-tree is odd $\left(k+2\left\lfloor\log_2 k\right\rfloor+3\right)$-colorable.
\end{theorem}

Let $f(k)$ be a function such that every $k$-tree is odd $(k+f(k))$-colorable.
Under this notation, we have $f(k)\leq 2\left\lfloor\log_2 k\right\rfloor+3$ from Theorem~\ref{thm:ktree}, which improves the bound of $f(k)\leq k+1$ obtained from Corollary~\ref{cor:known bound} in terms of the order of $k$.
As we mentioned above, there is a $k$-tree which is not odd $(k+1)$-colorable, and hence $f(k)$ is bounded by $2\leq f(k)\leq 2\left\lfloor\log_2 k\right\rfloor+3$.
Our second and third results show that the lower bound is attained when $k=2,3$.

\begin{theorem}\label{thm:2tree}
  Every 2-tree is odd $4$-colorable.
\end{theorem}

\begin{theorem}\label{thm:3tree}
  Every 3-tree is odd $5$-colorable.
\end{theorem}

Since every maximal outerplanar graph is a 2-tree, Theorem~\ref{thm:2tree} implies that every maximal outerplanar graph is odd $4$-colorable, which is a corollary of the result in~\cite{KMOOT2024+}.

Considering these results and the result that every tree is odd $3$-colorable in Caro, Petru\v{s}evski, and \v{S}krekovski~\cite{CPS2022}, we pose the following conjecture, which states that the lower bound of $f(k)$ is attained for every $k$.

\begin{conjecture}\label{conj:ktree}
  For every positive integer $k$, every $k$-tree is odd $(k+2)$-colorable.
\end{conjecture}

This paper is organized as follows; we give proofs of Theorems~\ref{thm:ktree} and \ref{thm:2tree} in Sections~\ref{section:ktree} and \ref{section:2tree}, respectively.
The proof of Theorem~\ref{thm:3tree} goes similarly with that of Theorem~\ref{thm:2tree} but it needs complicated case-analysis, thus we put it in the appendix.

Before we go to the next section, we give some easy observations that we use in our proofs.
For any graph $G$ and its proper coloring $\varphi$, it follows that $d_G(v)=\sum_{i=1}^k|\varphi^{-1}(i)\cap N_G(v)|$ for each vertex $v$ of $G$.
Thus, $v$ has an even degree if $|\varphi^{-1}(i)\cap N_G(v)|$ is even for all $i\in [k]$, which implies the following.

\begin{observation}\label{ob:odd deg vtx}
  Let $G$ be a graph, and let $v$ be a vertex of $G$.
  If the degree of $v$ is odd, then $v$ satisfies the odd condition with respect to any proper coloring of $G$.
\end{observation}

Since every vertex of a $k$-tree $G$ is contained in a clique of order $k+1$, for any proper coloring of $G$, every vertex has at least $k$ distinct colors that appear in its neighbors.
Thus, the following observation holds.

\begin{observation}\label{ob:small deg vtx}
  Let $G$ be a $k$-tree for some positive integer $k$, and let $v$ be a vertex of $G$.
  If $d_G(v)\leq 2k-1$, then $v$ satisfies the odd condition with respect to any proper coloring of $G$.
\end{observation}

\section{$k$-trees}\label{section:ktree}

\subsection{Auxiliary result}\label{subsection:ktree lemma}

The following lemma on a special ordering of vertices of a $k$-tree can easily be shown by an induction on the order.

\begin{lemma}\label{lem:good ordering}
    Let $k$ be a positive integer and let $G$ be a $k$-tree.
    Then there is an addition ordering $v_1,v_2,\dots ,v_n$ of $G$ such that $d_G(v_1)=k$.
\end{lemma}

Let $k$ be a positive integer and let $G$ be a $k$-tree of order $n$.
For a clique $V$ of order $k$ of $G$ and a vertex $u\in \bigcap_{v\in V}N_G(v)$, a \emph{branch} $B(V,u)$ of $G$ is a subgraph induced by $V\cup V(G_u)$, where $G_u$ is a component of $G-V$ that contains $u$.
Note that if $G-V$ is connected, then $\bigcap_{v\in V}N(v)=\{u\}$ for some vertex $u$ and $B(V,u)=G$.

By the definition of a $k$-tree, for any branch $B=B(V,u)$ of a $k$-tree $G$, the graph $G-(V(B)\setminus V)$ is a $k$-tree or isomorphic to the complete graph $K_k$.
Furthermore, the following lemma directly holds from the definition.

\begin{lemma}\label{lem:ktree branch order}
    Let $k$ be a positive integer and let $G$ be a $k$-tree of order $n$ with an addition ordering $v_1, v_2, \dots , v_n$.
    Let $B_i=B(V_i, v_i)$ be a branch of $G$ where $V_i=N_G(v_i)\cap \{v_j\mid 1\leq j\leq i-1\}$, and let $u_0,u_1,\dots u_s$ be an ordering of $V(B_i)\setminus V_i$ induced from the addition ordering of $G$.
    Then, 
    \begin{enumerate}[label=(\alph*)]
        \item $u_0=v_i$,
        \item $\bigl|N_B(u_j)\cap (V_i\cup\{u_\ell\mid 1\leq \ell\leq j-1\})\bigr|=k$ for every $j\in [s]$, and
        \item $\bigl|V_i\cap \bigcap_{\ell=1}^j N_B(u_\ell)\bigr|\geq \bigl|V_i\cap \bigcap_{\ell=1}^{j-1} N_B(u_\ell)\bigr|-1$ for every $j\in [s]$.
    \end{enumerate}
\end{lemma}

\subsection{Proof of Theorem~\ref{thm:ktree}}\label{subsection:ktree proof}

We fix a positive integer $k$ and let $r=\left\lfloor\log_2 k\right\rfloor+1$.
If $k\leq 6$, then we have $k+2\left\lfloor\log_2 k\right\rfloor+3\geq 2k+1$, and thus the statement follows from Corollary~\ref{cor:known bound}.
Hence, we assume that $k\geq 7$, which implies that $r+1<k$.

Let $G$ be a $k$-tree of order $n$.
The proof goes by induction on $n$.
If $n\leq k+2r+1$, then it is obvious that $G$ is odd $(k+2r+1)$-colorable.
We assume that $n>k+2r+1$, and that every $k$-tree of order less than $n$ is odd $(k+2r+1)$-colorable.

By Lemma~\ref{lem:good ordering}, let $v_1,v_2,\dots , v_n$ be an addition ordering of $G$ such that $d_G(v_1)=k$.
For each $i\in \{k+1, k+2, \dots , n\}$, let $V_i=N_G(v_i)\cap \{v_j\mid 1\leq j\leq i-1\}$ and let $B_i$ be the branch $B(V_i,v_i)$ of $G$.
Since $d_G(v_1)=k$, we have $B_{k+1}=B(\{v_1,v_2,\dots , v_k\}, v_{k+1})=G$, and thus $|V(B_{k+1})|>k+2r+1$.
Let $t$ be the largest index in $\{k+1,k+2,\dots , n\}$ such that $|V(B_t)|\geq k+r+1$.
Let $u_0, u_1, \dots , u_s$ be the ordering of $V(B_t)\setminus V_t$ with the conditions in Lemma~\ref{lem:ktree branch order}.
Note that $s\geq r$ by the choice of $B_t$.
Since $u_i$ comes after $u_0$ in the ordering $v_1, v_2, \dots v_n$, by the maximality of $t$, we have $d_G(u_i)=|N_G(u_i)|\leq|B_{u_i}|\leq k+r+1<2k$ for every $i\in\{2,3,\dots , s\}$.
For visibility, let $B:=B_t$ and let $V_t=\{w_1,w_2,\dots , w_k\}=:W$.

Let $G'=G- V(B)\setminus W$, $U=\{u_i\mid 1\leq i\leq r\}$, and $G''=G[V(G')\cup U]$.
We set $W_0=W\cap \bigcap_{i=1}^r N_G(u_i)$, and let $\overline{W}=W\setminus W_0$.
Without loss of generality, we may assume that $\overline{W}=\{w_1,w_2,\dots , w_\ell\}$ for some integer $\ell\geq 1$.
Since $G'$ is a $k$-tree of order less than $n$, by the induction hypothesis, $G'$ admits an odd $(k+2r+1)$-coloring $\varphi'$.
Without loss of generality, we may assume that $\varphi'(w_j)=j$ for each $j\in [k]$.
Let $\varphi(x)=\varphi'(x)$ for every vertex $x\in V(G')$.

\begin{claim}\label{claim:ktree injection}
  There exists an injection $\sigma: \overline{W}\to U$ such that $w\notin N_G(\sigma(w))$ for every $w\in \overline{W}$.
  In particular, $\ell\leq r$.
\end{claim}

\begin{proof}
  For each $w\in \overline{W}$, let $j_w\in [r]$ be the smallest index such that $w\notin N_G(u_{j_w})$, and define $\sigma(w)=u_{j_w}$.
  If $\sigma(w)=\sigma(w')=j$ for some $w, w'\in \overline{W}$, then it follows that $w, w'\in W\cap\bigcap_{\ell=1}^{j-1}N_G(u_\ell)$ and $w,w'\notin N_G(u_j)$.
  By Lemma~\ref{lem:ktree branch order}, we have $\bigl|W\cap \bigcap_{\ell=1}^{j}N_G(u_\ell)\bigr|\geq \bigl|W\cap \bigcap_{\ell=1}^{j-1}N_G(u_\ell)\bigr|-1$, which forces $w=w'$.
  Thus, $\sigma$ is a desired injection from $\overline{W}$ to $U$.
\end{proof}

Let $\sigma$ be an injection in Claim~\ref{claim:ktree injection}.
For each $i\in [r]$, we define a color $c_i$ by $c_i=j$ if $u_i=\sigma(v_j)$ for some $j\in [\ell]$, and $c_i=k+i$ otherwise.
Note that $c_i\neq c_{i'}$ for each pair of distinct $i,i'\in [r]$.
Let $C:=\{c_i\mid 1\leq i\leq r\}$.

For $i=1,2,\dots ,r$, we sequentially define a color $\varphi(u_i)$ and a set $W_i\subseteq W_0$ as follows.
Suppose that $W_{i-1}$ has been defined, and let 
\begin{align*}
  &W_{i-1}^o=\left\{w\in W_{i-1}\;\middle|\; |\varphi^{-1}(c_i)\cap N_{G'}(w)|\text{ is odd}\right\}\text{ and}\\
  &W_{i-1}^e=\left\{w\in W_{i-1}\;\middle|\; |\varphi^{-1}(c_i)\cap N_{G'}(w)|\text{ is even}\right\}.
\end{align*}
We define $W_i$ and $\varphi(u_i)$ as $W_i=W_{i-1}^o$ and $\varphi(u_i)=c_i$ if $|W_{i-1}^o|\leq |W_{i-1}^e|$, and $W_i=W_{i-1}^e$ and $\varphi(u_i)=k+r+i$ otherwise.
By the definition, it is easy to verify that the resulting coloring of $G''$ is proper.

\begin{claim}\label{claim:ktree odd color}
  For every vertex $w\in W_0$, there is a color $c(w)\in C$ such that $|\varphi^{-1}(c(w))\cap N_{G''}(w)|$ is odd.
\end{claim}

\begin{proof}
  By the definition of $W_i$, we have $|W_i|\leq \left\lfloor\frac{|W_{i-1}|}{2}\right\rfloor$ for each $i\in [r]$.
  Since $|W_0|\leq |W|=k$ and $r=\left\lfloor\log_2 k\right\rfloor+1> \log_2 k$, it follows that $|W_r|=0$.
  Thus, for every vertex $w\in W_0$, there is an index $i(w)\in [r]$ such that $w\in W_{i(w)-1}\setminus W_{i(w)}$.
  By the definition of $\varphi(u_{i(w)})$, $c_{i(w)}\in C$ and $|\varphi^{-1}(c_{i(w)})\cap N_{G''}(w)|$ is odd, as desired.
\end{proof}

Now we color the remaining vertices in $V(B)\setminus (W\cup U)$. 
Since $\{\varphi(x)\mid x\in N_G(u_0)\cap V(G'')\}\cup C\subseteq [k]\cup \{k+i\mid u_i\in U\setminus\sigma(\overline{W})\}\cup [k+r+1,k+2r]$, we have
\begin{align*}
    \bigl|\{\varphi(x)\mid x\in N_G(u_0)\cap V(G'')\}\cup C\bigr|\leq k+(r-|\overline{W}|)+r=k+2r-|\overline{W}|,
\end{align*}
and hence $\bigl|[k+2r+1]\setminus(\{\varphi(x)\mid x\in N_G(u_0)\cap V(G'')\}\cup C)\bigr|\geq |\overline{W}|+1$.
Thus, we choose a color in $[k+2r+1]\setminus(\{\varphi(x)\mid x\in N_G(u_0)\cap V(G'')\}\cup C)$ as $\varphi(u_0)$ so that every vertex of $\overline{W}$ satisfies the odd condition.

We sequentially choose colors for $u_{r+1},u_{r+2},\dots ,u_s$ as follows.
For each $i\in [r+1,s]$, let $X_i=N_G(u_i)\cap (\{u_j\mid 0\leq j\leq i-1\}\cup W)$ and let $X_i'=N_G(u_i)\cap (\{u_0\}\cup W)$.
By Lemma~\ref{lem:ktree branch order}, we know that $|X_i|=k$ for each $i\in [r+1,s]$.
Since $\varphi(w_j)=j\in C$ for every $w_j\in N_G(u_i)\cap \overline{W}$, we have
\begin{align*}
  \bigl|[k+2r+1]\setminus(\{\varphi(x)\mid x\in X_i\}\cup C)\bigr|
  &\geq (k+2r+1)-(k+r-|N_G(u_i)\cap \overline{W}|)\\
  &=|N_G(u_i)\cap \overline{W}|+r+1\geq |X_i'|+1.
\end{align*}
Hence, we can choose a color $\varphi(u_i)\in [k+2r+1]\setminus(\{\varphi(x)\mid x\in X_i\}\cup C)$ so that every vertex in $X_i'$ satisfies the odd condition. 

It is easy to see that $\varphi$ is a proper $(k+2r+1)$-coloring of $G$, and that every vertex in $V(G')\setminus W$ satisfies the odd condition with respect to $\varphi$.
By the choice of colors for $\{u_0,u_{r+1},u_{r+2}, \dots , u_s\}$, each vertex of $\overline{W}\cup \{u_0\}$ satisfies the odd condition with respect to $\varphi$.
Since $d_G(u_i)<2k$ for every $i\in [s]$, Observation~\ref{ob:small deg vtx} implies that every vertex of $\{u_i\mid 1\leq i\leq s\}$ satisfies the odd condition with respect to $\varphi$.
Furthermore, for each $w\in W_0$, since $|\varphi^{-1}(c(w))\cap N_{G''}(w)|$ is odd for some color $c(v)\in C$ by Claim~\ref{claim:ktree odd color} and $\varphi^{-1}(c(w))\cap (V(G)\setminus V(G''))=\emptyset$,
we infer that $|\varphi^{-1}(c(w))\cap N_{G}(w)|=|\varphi^{-1}(c(w))\cap N_{G''}(w)|$ is an odd integer.
Thus, $\varphi$ is an odd $(k+2r+1)$-coloring of $G$, which completes the proof of Theorem~\ref{thm:ktree}.

\section{2-trees}\label{section:2tree}

\subsection{Some special branches}

We define some special branches used in our proof.
Let $G$ be a 2-tree.
\begin{itemize}
  \item An \emph{ear} of $G$ is a branch $B(\{v_1,v_2\},u_0)$ of $G$ with 3 vertices $\{v_1,v_2,u_0\}$ such that $N_G(u_0)=\{v_1,v_2\}$.
  \item A \emph{hat} of $G$ is a branch $B(\{v_1,v_2\},u_0)$ of $G$ with 5 vertices $\{v_1,v_2,u_0,u_1,u_2\}$ such that $N_G(u_0)=\{v_1,v_2,u_1,u_2\}$, $N_G(u_1)=\{v_1,u_0\}$ and $N_G(u_2)=\{v_2,u_0\}$ (Figure~\ref{fig:odd 2tree hat}).
  A hat with the root $\{v_1,v_2\}$ is \emph{good} if $d_G(v_1)$ is equal to $4$ or an odd integer.
  \item A \emph{double hat} of $G$ is a branch $B(\{v_1,v_2\},u_0)$ of $G$ with 9 vertices $\{v_1,v_2\}\cup \{u_i\mid 0\leq i\leq 6\}$ such that $N_G(u_0)=\{v_1,v_2,u_1,u_2,u_4,u_5\}$, $N_G(u_1)=\{v_1,u_0,u_3,u_4\}$, $N_G(u_2)=\{v_2,u_0,u_5,u_6\}$, $N_G(u_3)=\{v_1,u_1\}$, $N_G(u_4)=\{u_0,u_1\}$, $N_G(u_5)=\{u_0,u_2\}$ and $N_G(u_6)=\{v_2,u_2\}$ (Figure~\ref{fig:odd 2tree double hat}).
\end{itemize}

\begin{figure}[h]
  \centering
  \begin{minipage}{0.4\columnwidth}
    \centering
    \begin{tikzpicture}[roundnode/.style={circle, draw=black,fill=white, minimum size=1.5mm, inner sep=0pt}]
      \node [roundnode] (v1) at (-1.5,0){};
      \node [roundnode] (v2) at (1.5,0){};
      \node [roundnode] (u1) at (0,1){};
      \node [roundnode] (u2) at (-1,1){};
      \node [roundnode] (u3) at (1,1){};
      
      \draw (v1)--(u2)--(u1)--(v1)--(v2)--(u3)--(u1)--(v2);
      \draw (v1)--(-2,0);
      \draw (v2)--(2,0);

      \node at (-1.5,-0.3){$v_1$};
      \node at (1.5,-0.3){$v_2$};
      \node at (0,1.3){$u_0$};
      \node at (-1,1.3){$u_1$};
      \node at (1,1.3){$u_2$};
    \end{tikzpicture}
    \caption{A hat of a 2-tree.}
    \label{fig:odd 2tree hat}
  \end{minipage}
  \begin{minipage}{0.5\columnwidth}
    \centering
    \begin{tikzpicture}[roundnode/.style={circle, draw=black,fill=white, minimum size=1.5mm, inner sep=0pt}]
      \node [roundnode] (v1) at (-3,0){};
      \node [roundnode] (v2) at (3,0){};
      \node [roundnode] (u1) at (0,0){};
      \node [roundnode] (u2) at (-1.5,1){};
      \node [roundnode] (u3) at (1.5,1){};
      \node [roundnode] (u4) at (-2.5,1){};
      \node [roundnode] (u5) at (-0.5,1){};
      \node [roundnode] (u6) at (0.5,1){};
      \node [roundnode] (u7) at (2.5,1){};
  
      \draw (v1)--(u4)--(u2)--(v1)--(u1)--(u5)--(u2)--(u1);
      \draw (v2)--(u7)--(u3)--(v2)--(u1)--(u6)--(u3)--(u1);
      \draw (v1)..controls (-1.5,-0.6) and (1.5,-0.6)..(v2);
      \draw (v1)--(-3.5,0);
      \draw (v2)--(3.5,0);

      \node at (-3,-0.3){$v_1$};
      \node at (3,-0.3){$v_2$};
      \node at (0,-0.3){$u_0$};
      \node at (-1.5,1.3){$u_1$};
      \node at (1.5,1.3){$u_2$};
      \node at (-2.5,1.3){$u_3$};
      \node at (-0.5,1.3){$u_4$};
      \node at (0.5,1.3){$u_5$};
      \node at (2.5,1.3){$u_6$};
    \end{tikzpicture}
    \caption{A double hat of a 2-tree.}
    \label{fig:odd 2tree double hat}
  \end{minipage}
\end{figure}

For a 2-tree $G$, we define two sets $\mathcal{H}^{(2)}(G)$ and $\mathcal{T}^{(2)}(G)$ of subgraphs as follows.
In the following definition, we use the notation $v_3$ as $v_1$.

Let $(a,b,c)$ be a tuple of non-negative integers with $1\leq a+b+c\leq 2$.
For a clique $\{v_1,v_2,u_0\}$ of $G$, the branch $B(\{v_1,v_2\},u_0)$ belongs to $\mathcal{H}^{(2)}_{a,b,c}(G)$ if 
\begin{itemize}
    \item $|N_G(u_0)\cap N_G(v_1)|\leq 2$, and if the equality holds, then there is a vertex $u_1$ such that $N_G(u_0)\cap N_G(v_1)=\{v_2,u_1\}$ and $B(\{v_1,u_0\},u_1)$ is either an ear, a hat or a double hat of $G$,
    \item $|N_G(u_0)\cap N_G(v_2)|\leq 2$, and if the equality holds, then there is a vertex $u_2$ such that $N_G(u_0)\cap N_G(v_2)=\{v_1,u_2\}$ and $B(\{v_2,u_0\},u_2)$ is either an ear, a hat or a double hat of $G$, and
    \item $\{B(\{v_1,u_0\},u_1), B(\{v_2,u_0\},u_2)\}$ consists of $a$ ears, $b$ hats and $c$ double hats.
\end{itemize}

Let $(a,b,c)$ be a tuple of non-negative integers with $a+b+c=2$.
For two cliques $\{v_1,v_2,u_0\}$ and $\{v_1,v_2,w_0\}$ sharing $\{v_1,v_2\}$, the subgraph of $G$ induced by $V(B(\{v_1,v_2\},u_0))\cup V(B(\{v_1,v_2\},w_0))$ belongs to $\mathcal{T}^{(2)}_{a,b,c}(G)$ if 
\begin{itemize}
  \item each of $B(\{v_1,v_2\},u_0)$ and $B(\{v_1,v_2\},w_0)$ is either an ear, a hat or a double hat of $G$, and
  \item $\{B(\{v_1,v_2\},u_0),B(\{v_1,v_2\},w_0)\}$ consists of $a$ ears, $b$ hats and $c$ double hats.
\end{itemize}
For a graph $T\in \mathcal{T}^{(2)}_{a,b,c}(G)$ induced by $V(B(\{v_1,v_2\},u_0))\cup V(B(\{v_1,v_2\},w_0))$, the set of vertices $\{v_1,v_2\}$ is called the \emph{root} of $T$.

Let $\mathcal{H}^{(2)}(G)=\bigcup_{1\leq a+b+c\leq 2}\mathcal{H}^{(2)}_{a,b,c}(G)$ and let $\mathcal{T}^{(2)}(G)=\bigcup_{a+b+c=2}\mathcal{T}^{(2)}_{a,b,c}(G)$.
Note that a hat of $G$ belongs to $\mathcal{H}^{(2)}_{2,0,0}(G)$, and a double hat of $G$ belongs to $\mathcal{H}^{(2)}_{0,2,0}(G)$.

The following lemma holds.

\begin{lemma}\label{lem:odd 2tree unavoidable}
  Let $G$ be a 2-tree of order at least $4$.
  Then there is a subgraph $H\in \mathcal{H}^{(2)}(G)\cup\mathcal{T}^{(2)}(G)$ that is neither a hat nor a double hat of $G$.
\end{lemma}

\begin{proof}
    Let $V_0$ be the set of vertices of degree $2$ of $G$.
    For $i=1,2,3$, we define a set of vertices $V_i$ as follows: 
    If $|V(G)\setminus\bigcup_{j=0}^{i-1}V_j|\geq 3$, then let $V_i$ be the set of vertices of $G-\bigcup_{j=0}^{i-1}V_j$ whose degrees in $G-\bigcup_{j=0}^{i-1}V_j$ are equal to $2$.
    If $|V(G)\setminus\bigcup_{j=0}^{i-1}V_j|=2$, then let $V_i=V(G)\setminus\bigcup_{j=0}^{i-1}V_j$.
    Otherwise, let $V_i=\emptyset$.
    By the definition, $V_0$, $V_1$, $V_2$, and $V_3$ are pairwise disjoint.

    The assumption $|V(G)|\geq 4$ forces $V_1\neq \emptyset$.
    If $|V(G)\setminus V_0|=2$, then the two vertices in $V(G)\setminus V_0$ form a root of a subgraph in $\mathcal{T}^{(2)}_{2,0,0}(G)$.
    Thus, we assume that $|V(G)\setminus V_0|\geq 3$.
    Let $v_1$ be a vertex in $V_1$ and let 
    $N_G(v_1)\setminus V_0=\{x_1,y_1\}$ such that $|V_0\cap N_G(v_1)\cap N_G(x_1)|\geq |V_0\cap N_G(v_1)\cap N_G(y_1)|$.
    If there are distinct vertices $u,w\in V_0$ with $N_G(u)=N_G(w)=\{v_1,x_1\}$, then $B(\{v_1,x_1\},u)\cup B(\{v_1,x_1\},w)\in \mathcal{T}^{(2)}_{2,0,0}(G)$.
    Thus, we may assume that $|V_0\cap N_G(v_1)\cap N_G(x_1)|=1$, which implies that $B(\{x_1,y_1\},v_1)\in \mathcal{H}^{(2)}_{1,0,0}(G)\cup \mathcal{H}^{(2)}_{2,0,0}(G)$.
    If $B(\{x_1,y_1\},v_1)$ is not a hat of $G$, then we are done.
    Thus, since the choice of $v_1\in V_1$ is arbitrary, we may assume that every vertex of $V_1$ is contained in a hat of $G$.

    Now we consider $V_2$.
    If $V(G)\setminus(V_0\cup V_1)=\emptyset$, then $G$ is isomorphic to a graph in Figure~\ref{fig:2tree unavoidable1}, and hence $\mathcal{H}^{(2)}_{0,1,0}(G)\neq \emptyset$.
    Suppose that $V_2\neq \emptyset$.
    If $|V(G)\setminus (V_0\cup V_1)|=2$, then the two vertices in $V(G)\setminus(V_0\cup V_1)$ form a root of a subgraph in $\mathcal{T}^{(2)}_{0,2,0}(G)$.
    Thus, we assume that $|V(G)\setminus (V_0\cup V_1)|\geq 3$.
    Let $v_2$ be a vertex of $V_2$, and let 
    $N_G(v_2)\setminus (V_0\cup V_1)=\{x_2,y_2\}$ such that $|(V_0\cup V_1)\cap N_G(v_2)\cap N_G(x_2)|\geq |(V_0\cup V_1)\cap N_G(v_2)\cap N_G(y_2)|$.
    If there are distinct vertices $u,w\in V_0\cup V_1$ such that $N_G(u)\cap (V(G)\setminus (V_0\cup V_1))=N_G(w)\cap (V(G)\setminus (V_0\cup V_1))=\{v_2,x_2\}$, then we are done since $B(\{v_2,x_2\},u)\cup B(\{v_2,x_2\},w)\in \mathcal{T}^{(2)}(G)$.
    Thus, we may assume that $|(V_0\cup V_1)\cap N_G(v_2)\cap N_G(x_2)|=1$, which implies that $B(\{x_2,y_2\},v_2)\in \mathcal{H}^{(2)}_{0,1,0}(G)\cup \mathcal{H}^{(2)}_{0,2,0}(G)\cup \mathcal{H}^{(2)}_{1,1,0}(G)$.
    If $B(\{x_1,y_1\},v_1)$ is not a double hat of $G$, then we are done.
    Thus, since the choice of $v_2\in V_2$ is arbitrary, we may assume that every vertex of $V_2$ is contained in a double hat of $G$.

    Finally, we consider $V_3$.
    If $V(G)\setminus(V_0\cup V_1\cup V_2)=\emptyset$, then $G$ is isomorphic to a graph in Figure~\ref{fig:2tree unavoidable2}, and hence $\mathcal{H}^{(2)}_{1,0,1}(G)\neq \emptyset$.
    Suppose that $V_3\neq \emptyset$.
    If $|V(G)\setminus (V_0\cup V_1\cup V_2)|=2$, then the two vertices in $V(G)\setminus(V_0\cup V_1\cup V_2)$ form a root of a subgraph in $\mathcal{T}^{(2)}_{0,0,2}(G)$.
    Thus, we assume that $|V(G)\setminus (V_0\cup V_1\cup V_2)|\geq 3$.
    Let $v_3$ be a vertex of $V_3$, and let 
    $N_G(v_3)\setminus (V_0\cup V_1\cup V_2)=\{x_3,y_3\}$ such that $|(V_0\cup V_1\cup V_2)\cap N_G(v_3)\cap N_G(x_3)|\geq |(V_0\cup V_1\cup V_2)\cap N_G(v_3)\cap N_G(y_3)|$.
    If there are distinct vertices $u,w\in V_0\cup V_1\cup V_2$ such that $N_G(u)\cap (V(G)\setminus (V_0\cup V_1\cup V_2))=N_G(w)\cap (V(G)\setminus (V_0\cup V_1\cup V_2))=\{v_3,x_3\}$, then we are done since $B(\{v_3,x_3\},u)\cup B(\{v_3,x_3\},w)\in \mathcal{T}^{(2)}(G)$.
    Thus, we may assume that $|(V_0\cup V_1\cup V_2)\cap N_G(v_3)\cap N_G(x_3)|=1$, which implies that $B(\{x_3,y_3\},v_3)\in \mathcal{H}^{(2)}_{0,0,1}(G)\cup \mathcal{H}^{(2)}_{0,0,2}(G)\cup \mathcal{H}^{(2)}_{0,1,1}(G)\cup \mathcal{H}^{(2)}_{1,0,1}(G)$.
    This completes the proof of Lemma~\ref{lem:odd 2tree unavoidable}.
\end{proof}

\begin{figure}
  \centering
  \begin{minipage}{0.4\columnwidth}
    \centering
    \begin{tikzpicture}[roundnode/.style={circle, draw=black,fill=white, minimum size=1.5mm, inner sep=0pt}]
      \node [roundnode] (v1) at (330:1){};
      \node [roundnode] (v2) at (90:1){};
      \node [roundnode] (v3) at (210:1){};
      \node [roundnode] (u1) at (30:1.5){};
      \node [roundnode] (u2) at (150:1.5){};
      \node [roundnode] (u3) at (270:1.5){};

      \draw (v1)--(u1)--(v2)--(u2)--(v3)--(u3)--(v1)--(v2)--(v3)--(v1);

      \node at (210:1.3){$u$};
      \node at (270:1.8){$v$};
      \node at (330:1.3){$w$};
    \end{tikzpicture}
    \subcaption{The case where $V(G)\setminus (V_0\cup V_1)=\emptyset$.
    It follows that $B(\{u,v\},w)\in \mathcal{H}^{(2)}_{0,1,0}(G)$.}
    \label{fig:2tree unavoidable1}
  \end{minipage}
  \hspace{0.1\columnwidth}
  \begin{minipage}{0.4\columnwidth}
    \centering
    \begin{tikzpicture}[roundnode/.style={circle, draw=black,fill=white, minimum size=1.5mm, inner sep=0pt}]
      \node [roundnode] (v1) at (330:1){};
      \node [roundnode] (v2) at (90:1){};
      \node [roundnode] (v3) at (210:1){};
      \node [roundnode] (u1) at (30:1.5){};
      \node [roundnode] (u2) at (150:1.5){};
      \node [roundnode] (u3) at (270:1.5){};
      \node [roundnode] (w1) at (0:1.5){};
      \node [roundnode] (w2) at (60:1.5){};
      \node [roundnode] (w3) at (120:1.5){};
      \node [roundnode] (w4) at (180:1.5){};
      \node [roundnode] (w5) at (240:1.5){};
      \node [roundnode] (w6) at (300:1.5){};

      \draw (v1)--(u1)--(v2)--(u2)--(v3)--(u3)--(v1)--(v2)--(v3)--(v1);
      \draw (v1)--(w1)--(u1)--(w2)--(v2)--(w3)--(u2)--(w4)--(v3)--(w5)--(u3)--(w6)--(v1);
      \node at (210:1.3){$u'$};
      \node at (270:1.8){$v'$};
      \node at (330:1.3){$w'$};
    \end{tikzpicture}
    \subcaption{The case where $V(G)\setminus (V_0\cup V_1\cup V_2)=\emptyset$.
    It follows that $B(\{u',v'\},w')\in \mathcal{H}^{(2)}_{1,0,1}(G)$.}
    \label{fig:2tree unavoidable2}
  \end{minipage}
  \caption{Special graphs in the proof of Lemma~\ref{lem:odd 2tree unavoidable}.}
\end{figure}
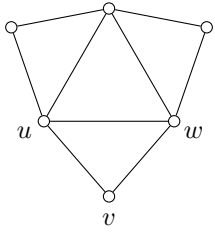
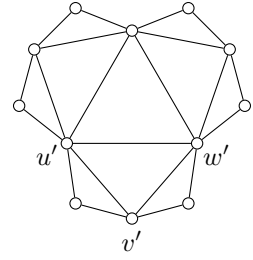

The following lemma is used repeatedly in our proof of Theorem~\ref{thm:2tree}.

\begin{lemma}\label{lem:odd 2tree nearodd}
  Let $G$ be a 2-tree.
  Suppose that $G$ has a branch $B=B(\{v_1,v_2\},u_0)$ which is either an ear, a hat, or a double hat.
  Let $G'=G-(V(B)\setminus\{v_1,v_2\})$.
  If $G'$ admits an odd 4-coloring $\varphi'$, then $\varphi'$ can be extended to a proper 4-coloring of $G$ such that every vertex of $V(G)$ other than $v_1$ satisfies the odd condition.
\end{lemma}

\begin{proof}
  Let $\varphi'$ be an odd 4-coloring of $G'$.
  Without loss of generality, we may assume that $\varphi'(v_1)=1$ and $\varphi'(v_2)=2$.
  For each vertex $v\in V(G')$, let $\varphi(v)=\varphi'(v)$.
  If $B$ is an ear, then we choose $\varphi(u_0)\in \{3,4\}$ so that $v_2$ satisfies the odd condition.
  If $B$ is a hat, then let $V(B)=\{v_1,v_2,u_0,u_1,u_2\}$ such that $N_G(u_0)=\{v_1,v_2,u_1,u_2\}$, $N_G(u_1)=\{v_1,u_0\}$ and $N_G(u_2)=\{v_2,u_0\}$.
  We let $\varphi(u_0)=3$, choose $\varphi(u_2)\in \{1,4\}$ so that $v_2$ satisfies the odd condition, and choose $\varphi(u_1)\in\{2,4\}$ so that $u_0$ satisfies the odd condition.
  If $B$ is a double hat, then let $V(B)=\{v_1,v_2\}\cup \{u_i\mid 0\leq i\leq 6\}$ such that $N_G(u_0)=\{v_1,v_2,u_1,u_2,u_4,u_5\}$, $N_G(u_1)=\{v_1,u_0,u_3,u_4\}$, $N_G(u_2)=\{v_2,u_0,u_5,u_6\}$, $N_G(u_3)=\{v_1,u_1\}$, $N_G(u_4)=\{u_0,u_1\}$, $N_G(u_5)=\{u_0,u_2\}$, and $N_G(u_6)=\{v_2,u_2\}$.
  We first let $\varphi(u_0)=3$ and $\varphi(u_1)=\varphi(u_2)=4$.
  Then we choose $\varphi(u_6)\in\{1,3\}$ so that $v_2$ satisfies the odd condition, choose $\varphi(u_5)\in\{1,2\}$ so that $u_2$ satisfies the odd condition, choose $\varphi(u_4)\in \{1,2\}$ so that $u_0$ satisfies the odd condition, and choose $\varphi(u_3)\in \{2,3\}$ so that $u_1$ satisfies the odd condition.
  For each case, it is easy to verify that $\varphi$ is a desired coloring of $G$.
\end{proof}

\subsection{Proof of Theorem \ref{thm:2tree}}

\setcounter{case}{0}

Let $G$ be a 2-tree of order $n$.
The proof goes by induction on $n$.
If $n\leq 4$, then trivially $G$ is odd $4$-colorable.
Hence, we assume that $n\geq 5$ and every 2-tree of order less than $n$ is odd $4$-colorable.
Note that every vertex of $G$ of degree $2$ satisfies the odd condition with respect to any proper coloring of $G$ by Observation~\ref{ob:small deg vtx}.

By Lemma~\ref{lem:odd 2tree unavoidable}, there is a subgraph $H\in\mathcal{H}^{(2)}(G)\cup\mathcal{T}^{(2)}(G)$ which is neither a hat nor a double hat of $G$.
We show that $H$ is a reducible structure with respect to odd 4-coloring.
In order to deal with some easy cases together, we first consider the case we can find a good hat $H'$ in $H$.

\begin{case}\label{case:odd 2tree good hat}
  $G$ has a good hat $H'$. 
\end{case}

Let $H'$ be a hat with $5$ vertices $\{v_1,v_2,u_0,u_1,u_2\}$ such that $\{v_1,v_2\}$ is the root of $H'$, $N_G(u_0)=\{v_1,v_2,u_1,u_2\}$, $N_G(u_1)=\{v_1,u_0\}$, $N_G(u_2)=\{v_2,u_0\}$, and $d_G(v_1)$ is equal to $4$ or an odd integer.
Let $G'=G-\{u_0,u_1,u_2\}$.
By the induction hypothesis, $G'$ admits an odd 4-coloring $\varphi$.
Without loss of generality, we may assume that $\varphi(v_1)=1$ and $\varphi(v_2)=2$.
If $d_G(v_1)$ is odd, then we may extend $\varphi$ to a proper 4-coloring of $G$ such that every vertex other than $v_1$ satisfies the odd condition by Lemma~\ref{lem:odd 2tree nearodd}, and $v_1$ satisfies the odd condition by Observation~\ref{ob:odd deg vtx}.
Hence, we assume that $d_G(v_1)=4$.
We let $\varphi(u_0)=3$, $\varphi(u_1)=4$, and choose a color in $\{1,4\}$ as $\varphi(u_2)$ so that $v_2$ satisfies the odd condition.
By the choice of colors, every vertex other than $\{v_1,u_0\}$ satisfies the odd condition.
Since $d_G(v_1)=4$ and the three vertices $v_2,u_0,u_1\in N_G(v_1)$ receive three distinct colors, $v_1$ satisfies the odd condition with respect to $\varphi$.
Similarly, since $d_G(u_0)=4$ and $v_1,v_2,u_1\in N_G(u_0)$, $u_0$ satisfies the odd condition with respect to $\varphi$.
Thus, $\varphi$ is an odd 4-coloring of $G$, as desired.

\vspace{\baselineskip}
If $H\in \mathcal{H}^{(2)}_{0,1,0}(G)\cup \mathcal{H}^{(2)}_{1,1,0}(G)\cup\mathcal{H}^{(2)}_{0,0,1}(G)\cup\mathcal{H}^{(2)}_{0,1,1}(G)$, then $H$ contains a good hat of $G$.
Thus, by Case~\ref{case:odd 2tree good hat}, we may assume that $H\notin \mathcal{H}^{(2)}_{0,1,0}(G)\cup \mathcal{H}^{(2)}_{1,1,0}(G)\cup\mathcal{H}^{(2)}_{0,0,1}(G)\cup\mathcal{H}^{(2)}_{0,1,1}(G)$.

In each of the following cases, let $G'=G-(V(H)\setminus\{v_1,v_2\})$, where $\{v_1,v_2\}$ is the root of $H$.
By the induction hypothesis, $G'$ admits an odd 4-coloring $\varphi$.
Without loss of generality, we may assume that $\varphi(v_1)=1$ and $\varphi(v_2)=2$.
We extend $\varphi$ to $G$ as in the following five cases.

\begin{case}\label{case:odd 2tree h100}
  $H\in\mathcal{H}^{(2)}_{1,0,0}(G)$.
\end{case}

Let $H=G[\{v_1,v_2,u_0,u_1\}]$ such that $N_G(u_0)=\{v_1,v_2,u_1\}$ and $N_G(u_1)=\{v_1,u_0\}$.
We choose $\varphi(u_0)\in \{3,4\}$ so that $v_2$ satisfies the odd condition, and then choose $\varphi(u_1)\in\{2,3,4\}\setminus\{\varphi(u_0)\}$ so that $v_1$ satisfies the odd condition.
It is easy to verify that $\varphi$ is an odd 4-coloring of $G$.

\begin{case}\label{case:odd 2tree h101}
  $H\in\mathcal{H}^{(2)}_{1,0,1}(G)\cup \mathcal{H}^{(2)}_{0,0,2}(G)$.
\end{case}

Let $H=B(\{v_1,v_2\},u_0)$ such that $B(\{v_1,u_0\},u_1)$ is a double hat and $B(\{v_2,u_0\},u_2)$ is either an ear or a double hat (Figure~\ref{fig:odd 2tree h101}).
Let $V(B(\{v_1,u_0\},u_1))=\{v_1,u_0,u_1,u_3,u_4,u_5,u_6,u_7,u_8\}$ such that $N_G(u_1)=\{v_1,u_0,u_3,u_4,u_6,u_7\}$, $N_G(u_3)=\{v_1,u_1,u_5,u_6\}$, and $N_G(u_4)=\{u_0,u_1,u_7,u_8\}$.
Let $G'=G-V(H)\setminus\{v_1,v_2\}$.
We first let $\varphi(u_1)=2$, $\varphi(u_3)=3$, $\varphi(u_5)=4$, and choose $\varphi(u_0)\in \{3,4\}$ so that $v_1$ satisfies the odd condition.
Then we sequentially use Lemma~\ref{lem:odd 2tree nearodd} to $B(\{u_0,v_2\},u_2)$, $B(\{u_1,u_0\},u_4)$, and $B(\{u_3,u_1\},u_6)$ to obtain a proper 4-coloring of $G$ such that every vertex other than $u_3$ satisfies the odd condition.
Since $d_G(u_3)=4$ and three colors appear in the neighbors of $u_3$, $u_3$ also satisfies the odd condition.
Thus, $\varphi$ is an odd 4-coloring of $G$.

\begin{figure}
  \centering
  \begin{tikzpicture}[roundnode/.style={circle, draw=black,fill=white, minimum size=1.5mm, inner sep=0pt}]
    \node [roundnode] (v1) at (-4,0){};
    \node [roundnode] (v2) at (3,0){};
    \node [roundnode] (u0) at (0,0){};
    \node [roundnode] (u1) at (-2,0){};
    \node [roundnode] (u3) at (-3,1){};
    \node [roundnode] (u4) at (-1,1){};
    \node [roundnode] (u5) at (-3.5,1){};
    \node [roundnode] (u6) at (-2.5,1){};
    \node [roundnode] (u7) at (-1.5,1){};
    \node [roundnode] (u8) at (-0.5,1){};

    \draw (v1)--(u5)--(u3)--(v1)--(u1)--(u6)--(u3)--(u1)--(u7)--(u4)--(u1)--(u0)--(u8)--(u4)--(u0)--(v2);
    \draw (v1)..controls (-3,-0.6) and (-1,-0.6)..(u0);
    \draw (v1)..controls (-2.6,-1) and (0,-1)..(v2);
    \filldraw[fill=lightgray] (u0)..controls (1,0.5) and (2,0.5) ..(v2);
    \draw (v1)--(-4.5,0);
    \draw (v2)--(3.5,0);

    \node at (-4,-0.3){$v_1$};
    \node at (3,-0.3){$v_2$};
    \node at (0,-0.3){$u_0$};
    \node at (-2,-0.3){$u_1$};
    \node at (1.5,0.8){$B(\{v_2,u_0\},u_2)$};
    \node at (-3,1.3){$u_3$};
    \node at (-1,1.3){$u_4$};
    \node at (-3.5,1.3){$u_5$};
    \node at (-2.5,1.3){$u_6$};
    \node at (-1.5,1.3){$u_7$};
    \node at (-0.5,1.3){$u_8$};
  \end{tikzpicture}
  \caption{Case \ref{case:odd 2tree h101}}
  \label{fig:odd 2tree h101}
\end{figure}

\begin{case}\label{case:odd 2tree t200}
  $H\in\mathcal{T}^{(2)}_{2,0,0}(G)$.
\end{case}

Let $V(H)=\{v_1,v_2,u_0,w_0\}$ such that $N_G(u_0)=N_G(w_0)=\{v_1,v_2\}$.
By setting $\varphi(u_0)=\varphi(w_0)=3$, the parity of the number of neighbors with each color around $v_1$ and $v_2$ does not change.
Thus, $\varphi$ is an odd 4-coloring of $G$.

\begin{case}\label{case:odd 2tree t010}
  $H\in\mathcal{T}^{(2)}_{a,b,c}(G)$ for some tuple $(a,b,c)$ with $a+b+c=2$ and $b\geq 1$.
\end{case}

Let $B_1=B(\{v_1,v_2\},u_0)$ and $B_2=B(\{v_1,v_2\},w_0)$ be branches such that $V(H)=V(B_1)\cup V(B_2)$ and $B_1$ is a hat of $G$.
Let $V(B_1)=\{v_1,v_2,u_0,u_1,u_2\}$ such that $N_G(u_0)=\{v_1,v_2,u_1,u_2\}$, $N_G(u_1)=\{v_1,u_0\}$, and $N_G(u_2)=\{v_2,u_0\}$.
We first let $\varphi(u_0)=3$ and $\varphi(u_1)=4$.
By applying Lemma~\ref{lem:odd 2tree nearodd}, we choose colors for $V(B_2)\setminus\{v_1,v_2\}$ so that every vertex of $V(B_2)\setminus\{v_2\}$ satisfies the odd condition, and choose $\varphi(u_2)\in \{1,4\}$ so that $v_2$ satisfies the odd condition.
By the choice of colors, every vertex other than $u_0$ satisfies the odd condition.
Since $d_G(u_0)=4$ and three colors appear in the neighbors of $u_0$, $u_0$ satisfies the odd condition and thus $\varphi$ is an odd 4-coloring of $G$.

\begin{case}\label{case:odd 2tree t001}
  $H\in\mathcal{T}^{(2)}_{a,b,c}(G)$ for some tuple $(a,b,c)$ with $a+b+c=2$ and $c\geq 1$.
\end{case}

Let $B_1=B(\{v_1,v_2\},u_0)$ and $B_2=B(\{v_1,v_2\},w_0)$ be branches such that $V(H)=V(B_1)\cup V(B_2)$ and $B_1$ is a double hat of $G$.
Let $V(B_1)=\{v_1,v_2\}\cup\{u_i\mid 0\leq i\leq 6\}$ such that $N_G(u_0)=\{v_1,v_2,u_1,u_2,u_4,u_5\}$, $N_G(u_1)=\{v_1,u_0,u_3,u_4\}$, $N_G(u_2)=\{v_2,u_0,u_5,u_6\}$, $N_G(u_3)=\{v_1,u_1\}$, $N_G(u_4)=\{u_0,u_1\}$, $N_G(u_5)=\{u_0,u_2\}$, and $N_G(u_6)=\{v_2,u_2\}$.
We first let $\varphi(u_0)=3$, $\varphi(u_1)=4$, and $\varphi(u_3)=2$.
By applying Lemma~\ref{lem:odd 2tree nearodd} to $B_2$ and $B(\{u_0,v_2\},u_2)$, we can color $V(H)\setminus\{v_1,v_2,u_4\}$ so that every vertex other than $u_0$ and $u_1$ satisfies the odd condition, and finally we choose $\varphi(u_4)\in\{1,2\}$ so that $u_0$ satisfies the odd condition.
Since $d_G(u_1)=4$ and three colors appear in the neighbors of $u_1$, $u_1$ satisfies the odd condition.
Thus, $\varphi$ is an odd 4-coloring of $G$.
This completes the proof of Theorem~\ref{thm:2tree}.

\section*{Acknowledgement}
This work partially depends on discussion while the authors visited Zhejiang Normal University in China. The authors thank Xuding Zhu for the invitation and valuable comments.
Masaki Kashima is supported by JSPS KAKENHI Grant Number 25KJ2077.
Kenta Ozeki is supported by JSPS KAKENHI, Grant Numbers 19H01803, 20H5795, and 22K19773, 23K03195.

\newpage
\section*{Appendix}

We give a proof of Theorem~\ref{thm:3tree}. 

\subsection{Some special branches}

We define some special branches of a 3-tree.
Let $G$ be a 3-tree.

\begin{itemize}
  \item An \emph{ear} of $G$ is a branch $B(\{v_1,v_2,v_3\},u_0)$ of $G$ with $4$ vertices $\{v_1,v_2,v_3,u_0\}$ such that $N_G(u_0)=\{v_1,v_2,v_3\}$.
  \item A \emph{one-hat} of $G$ is a branch $B(\{v_1,v_2,v_3\},u_0)$ of $G$ with $5$ vertices $\{v_1,v_2,v_3,u_0,u_1\}$ such that $N_G(u_0)=\{v_1,v_2,v_3,u_1\}$ and $N_G(u_1)=\{v_1,v_2,u_0\}$ (Figure~\ref{fig:odd 3tree 1hat}).
  \item A \emph{one-hat plus} of $G$ is a branch $B(\{v_1,v_2,v_3\},u_0)$ of $G$ with $7$ vertices $\{v_1,v_2,v_3\}\cup\{u_i\mid 0\leq i\leq 3\}$ such that $N_G(u_0)=\{v_1,v_2,v_3,u_1,u_2,u_3\}$, $N_G(u_1)=\{v_1,v_2,u_0,u_3\}$, $N_G(u_2)=\{v_2,v_3,u_0\}$ and $N_G(u_3)=\{v_1,u_0,u_1\}$ (Figure~\ref{fig:odd 3tree 1hat plus}).
\end{itemize}

\begin{figure}[h]
  \centering
  \begin{minipage}{0.45\columnwidth}
    \centering
    \begin{tikzpicture}[roundnode/.style={circle, draw=black,fill=white, minimum size=1.5mm, inner sep=0pt}]
      \node [roundnode] (v1) at (90:2.4){};
      \node [roundnode] (v2) at (210:2.4){};
      \node [roundnode] (v3) at (330:2.4){};
      \node [roundnode] (u0) at (0:0){};
      \node [roundnode] (u1) at (150:0.6){};

      \foreach \x in {1,2,3} \draw (u0)--(v\x);
      \draw (v1)--(v2)--(v3)--(v1);
      \foreach \x in {1,2} \draw (u1)--(v\x);
      \draw (u0)--(u1);
      \draw (v1)--(90:2.7);
      \draw (v2)--(210:2.7);
      \draw (v3)--(330:2.7);

      \node at (97:2.4){$v_1$};
      \node at (217:2.4){$v_2$};
      \node at (323:2.4){$v_3$};
      \node at (30:0.3){$u_0$};
      \node at (150:0.9){$u_1$};
    \end{tikzpicture}
    \caption{A one-hat of a 3-tree.}
    \label{fig:odd 3tree 1hat}
  \end{minipage}
  \begin{minipage}{0.45\columnwidth}
    \centering
    \begin{tikzpicture}[roundnode/.style={circle, draw=black,fill=white, minimum size=1.5mm, inner sep=0pt}]
      \node [roundnode] (v1) at (90:2.4){};
      \node [roundnode] (v2) at (210:2.4){};
      \node [roundnode] (v3) at (330:2.4){};
      \node [roundnode] (u0) at (0:0){};
      \node [roundnode] (u1) at (180:0.9){};
      \node [roundnode] (u2) at (270:0.6){};
      \node [roundnode] (u3) at (120:0.7){};

      \foreach \x in {1,2,3}{\draw (u0)--(v\x);
      \draw (u0)--(u\x);}
      \draw (v1)--(v2)--(v3)--(v1);
      \foreach \x in {1,2} \draw (u1)--(v\x);
      \foreach \x in {0,1} \draw (u3)--(u\x);
      \draw (u3)--(v1);
      \foreach \x in {2,3} \draw (u2)--(v\x);
      \draw (v1)--(90:2.7);
      \draw (v2)--(210:2.7);
      \draw (v3)--(330:2.7);

      \node at (97:2.4){$v_1$};
      \node at (217:2.4){$v_2$};
      \node at (323:2.4){$v_3$};
      \node at (30:0.3){$u_0$};
      \node at (197:0.9){$u_1$};
      \node at (148:0.45){$u_3$};
      \node at (270:0.9){$u_2$};
    \end{tikzpicture}
    \caption{A one-hat plus of a 3-tree.}
    \label{fig:odd 3tree 1hat plus}
  \end{minipage}
\end{figure}

For a 3-tree $G$, we define two sets $\mathcal{H}^{(3)}(G)$ and $\mathcal{T}^{(3)}(G)$ of subgraphs as follows.

Let $(a,b,c)$ be a tuple of non-negative integers with $1\leq a+b+c\leq 3$.
For a clique $\{v_1,v_2,v_3,u_0\}$ of $G$, the branch $B(\{v_1,v_2,v_3\},u_0)$ belongs to $\mathcal{H}^{(3)}_{a,b,c}(G)$ if
\begin{itemize}
    \item $|N_G(u_0)\cap N_G(v_1)\cap N_G(v_2)|\leq 2$, and if the equality holds, then there is a vertex $u_1$ such that $N_G(u_0)\cap N_G(v_1)\cap N_G(u_2)=\{v_3,u_1\}$ and $B(\{v_1,v_2,u_0\},u_1)$ is either an ear, a one-hat or a one-hat plus of $G$,
    \item $|N_G(u_0)\cap N_G(v_2)\cap N_G(v_3)|\leq 2$, and if the equality holds, then there is a vertex $u_2$ such that $N_G(u_0)\cap N_G(v_2)\cap N_G(u_3)=\{v_1,u_2\}$ and $B(\{v_2,v_3,u_0\},u_2)$ is either an ear, a one-hat or a one-hat plus of $G$,
    \item $|N_G(u_0)\cap N_G(v_3)\cap N_G(v_1)|\leq 2$, and if the equality holds, then there is a vertex $u_3$ such that $N_G(u_0)\cap N_G(v_3)\cap N_G(u_1)=\{v_2,u_3\}$ and $B(\{v_3,v_1,u_0\},u_3)$ is either an ear, a one-hat or a one-hat plus of $G$, and 
    \item $\{B(\{v_1,v_2,u_0\},u_1), B(\{v_2,v_3,u_0\},u_2), B(\{v_3,v_1,u_0\},u_3)\}$ consists of $a$ ears, $b$ one-hats and $c$ one-hat pluses.
\end{itemize}

Let $(a,b,c)$ be a tuple of non-negative integers with $a+b+c=2$.
For two cliques $\{v_1,v_2,v_3,u_0\}$ and $\{v_1,v_2,v_3,w_0\}$ sharing $\{v_1,v_2,v_3\}$, the subgraph of $G$ induced by $V(B(\{v_1,v_2,v_3\},u_0))\cup V(B(\{v_1,v_2,v_3\},w_0))$ belongs to $\mathcal{T}^{(3)}_{a,b,c}(G)$ if
\begin{itemize}
  \item each of $B(\{v_1,v_2,v_3\},u_0)$ and $B(\{v_1,v_2,v_3\},w_0)$ is either an ear, a one-hat, or a one-hat plus of $G$, and 
  \item $\{B(\{v_1,v_2,v_3\},u_0), B(\{v_1,v_2,v_3\},w_0)\}$ consists of $a$ ears, $b$ one-hats and $c$ one-hat pluses.
\end{itemize}
For a graph $T\in\mathcal{T}^{(3)}_{a,b,c}(G)$ induced by $V(B(\{v_1,v_2,v_3\},u_0))\cup V(B(\{v_1,v_2,v_3\},w_0))$, the set of vertices $\{v_1,v_2,v_3\}$ is called the \emph{root} of $T$.

Let $\mathcal{H}^{(3)}(G)=\bigcup_{1\leq a+b+c\leq 3}\mathcal{H}^{(3)}_{a,b,c}(G)$ and let $\mathcal{T}^{(3)}(G)=\bigcup_{a+b+c=2}\mathcal{T}^{(3)}_{a,b,c}(G)$.
Note that a one-hat of $G$ belongs to $\mathcal{H}^{(3)}_{1,0,0}(G)$ and a one-hat plus of $G$ belongs to $\mathcal{H}^{(3)}_{1,1,0}(G)$.

The following lemma holds.

\begin{lemma}\label{lem:3tree unavoidable}
  Let $G$ be a 3-tree with at least $5$ vertices. 
  Then there is a subgraph $H\in\mathcal{H}^{(3)}(G)\cup\mathcal{T}^{(3)}(G)$ which is neither a one-hat nor a one-hat plus of $G$. 
\end{lemma}

\begin{proof}
    Let $V_0$ be the set of vertices of degree $3$ of $G$.
    For $i=1,2,3$, we define a set of vertices $V_i$ as follows: 
    If $|V(G)\setminus\bigcup_{j=0}^{i-1}V_j|\geq 4$, then let $V_i$ be the set of vertices of $G-\bigcup_{j=0}^{i-1}V_j$ whose degrees in $G-\bigcup_{j=0}^{i-1}V_j$ are equal to $3$.
    If $|V(G)\setminus\bigcup_{j=0}^{i-1}V_j|=3$, then let $V_i=V(G)\setminus\bigcup_{j=0}^{i-1}V_j$.
    Otherwise, let $V_i=\emptyset$.
    By the definition, $V_0$, $V_1$, $V_2$, and $V_3$ are pairwise disjoint.
    
    The assumption $|V(G)|\geq 5$ forces $V_1\neq \emptyset$.
    If $|V(G)\setminus V_0|=3$, then the three vertices in $V(G)\setminus V_0$ form a root of a subgraph in $\mathcal{T}^{(3)}_{2,0,0}(G)$.
    Thus, we assume that $|V(G)\setminus V_0|\geq 4$.
    Let $v_1$ be a vertex in $V_1$ and let $N_G(v_1)\setminus V_0=\{x_1,y_1,z_1\}$ such that $|N_G(x_1)\cap N_G(y_1)\cap N_G(v_1)\cap V_0|\geq |N_G(y_1)\cap N_G(z_1)\cap N_G(v_1)\cap V_0|\geq |N_G(z_1)\cap N_G(x_1)\cap N_G(v_1)\cap V_0|$.
    If there are two distinct vertices $u,w\in V_0$ such that $N_G(u)=N_G(w)=\{v_1,x_1,y_1\}$, then $B(\{v_1,x_1,y_1\},u)\cup B(\{v_1,x_1,y_1\},w)\in \mathcal{T}^{(3)}_{2,0,0}(G)$.
    Thus, we may assume that $|N_G(x_1)\cap N_G(y_1)\cap N_G(v_1)\cap V_0|=1$, which implies that $B(\{x_1,y_1,z_1\},v_1)\in \mathcal{H}^{(3)}_{a,0,0}(G)$ for some $a\in\{1,2,3\}$.
    If $B(\{x_1,y_1,z_1\},v_1)$ is not a one-hat of $G$, then we are done.
    Thus, since the choice of $v_1\in V_1$ is arbitrary, we may assume that every vertex of $V_1$ is contained in a one-hat of $G$.

    Since every vertex of $V_1$ belongs to a one-hat of $G$, if $V(G)=V_0\cup V_1$, then $G$ is isomorphic to a graph obtained from $K_4$ by adding one vertex of degree $3$ to a triangle, and hence $\mathcal{T}^{(3)}_{2,0,0}(G)\neq \emptyset$.
    Suppose that $V_2\neq \emptyset$.
    If $|V(G)\setminus (V_0\cup V_1)|=3$, then the three vertices in $V(G)\setminus (V_0\cup V_1)$ form a root of a subgraph in $\mathcal{T}^{(3)}_{0,2,0}(G)$.
    Thus, we assume that $|V(G)\setminus (V_0\cup V_1)|\geq 4$.
    Let $v_2$ be a vertex of $V_2$ and let $N_G(v_2)\setminus(V_0\cup V_1)=\{x_2,y_2,z_2\}$ such that $|N_G(x_2)\cap N_G(y_2)\cap N_G(v_2)\cap (V_0\cup V_1)|\geq |N_G(y_2)\cap N_G(z_2)\cap N_G(v_2)\cap (V_0\cup V_1)|\geq |N_G(z_2)\cap N_G(x_2)\cap N_G(v_2)\cap (V_0\cup V_1)|$.
    If there are two distinct vertices $u,w\in V_0\cup V_1$ such that $N_G(u)\setminus (V_0\cup V_1)=N_G(w)\setminus (V_0\cup V_1)=\{v_2,x_2,y_2\}$, then $B(\{v_2,x_2,y_2\},u)\cup B(\{v_2,x_2,_2\},w)\in \mathcal{T}^{(3)}(G)$.
    Thus, we may assume that $|N_G(x_2)\cap N_G(y_2)\cap N_G(v_2)\cap (V_0\cup V_1)|=1$, which implies that $B(\{x_2,y_2,z_2\},v_2)\in\mathcal{H}^{(3)}_{a,b,0}(G)$ for some $a\geq 0$ and $b\geq 1$ with $a+b\leq 3$.
    If $B(\{x_2,y_2,z_2\},v_2)$ is not a one-hat plus of $G$, then we are done.
    Thus, since the choice of $v_2\in V_2$ is arbitrary, we may assume that $V_2\neq \emptyset$ and every vertex of $V_2$ is contained in a one-hat plus of $G$.
    Then it is easy to verify that $V_3\neq \emptyset$.

    If $|V(G)\setminus (V_0\cup V_1\cup V_2)|=3$, then the three vertices in $V(G)\setminus (V_0\cup V_1\cup V_2)$ form a root of a subgraph in $\mathcal{T}^{(3)}_{0,0,2}(G)$.
    Thus, we assume that $|V(G)\setminus (V_0\cup V_1\cup V_2)|\geq 4$.
    Let $v_3\in V_3$ and let $N_G(v_3)\setminus (V_0\cup V_1\cup V_2)=\{x_3,y_3,z_3\}$ such that $|N_G(x_3)\cap N_G(y_3)\cap N_G(v_3)\cap (V_0\cup V_1\cup V_2)|\geq |N_G(y_3)\cap N_G(z_3)\cap N_G(v_3)\cap (V_0\cup V_1\cup V_2)|\geq |N_G(z_3)\cap N_G(x_3)\cap N_G(v_3)\cap (V_0\cup V_1\cup V_2)|$.
    If there are two distinct vertices $u,w\in V_0\cup V_1\cup V_2$ such that $N_G(u)\setminus (V_0\cup V_1\cup V_2)=N_G(w)\setminus (V_0\cup V_1\cup V_2)=\{v_3,x_3,y_3\}$, then $B(\{v_3,x_3,y_3\},u)\cup B(\{v_3,x_3,y_3\},w)\in \mathcal{T}^{(3)}(G)$.
    Thus, we may assume that $|N_G(x_3)\cap N_G(y_3)\cap N_G(v_3)\cap (V_0\cup V_1\cup V_2)|=1$, which implies that $B(\{x_3,y_3,z_3\},v_3)\in\mathcal{H}^{(3)}_{a,b,c}(G)$ for some $a\geq 0$, $b\geq 0$ and $c\geq 1$ with $a+b+c\leq 3$.
    This completes the proof of Lemma~\ref{lem:3tree unavoidable}.
\end{proof}

The following lemma is used repeatedly in our proof of Theorem~\ref{thm:3tree}.

\begin{lemma}\label{lem:3tree nearodd}
  Let $G$ be a 3-tree, and suppose that $G$ has a branch $B=B(\{v_1,v_2,v_3\},u_0)$ which is either an ear, a one-hat, or a one-hat plus of $G$.
  Suppose that $G-(V(B)\setminus\{v_1,v_2,v_3\})$ admits an odd 5-coloring $\varphi'$.
  For any index $i\in \{1,2,3\}$, $\varphi'$ can be extended to a proper 5-coloring of $G'\cup B$ such that every vertex of $V(G)\setminus\{v_i,v_{i+1}\}$ satisfies the odd condition, where $v_4=v_1$.
  In particular, when $B$ is either a one-hat or a one-hat plus, then $\varphi'$ can be extended to a proper 5-coloring of $G'\cup B$ such that every vertex of $V(G)\setminus\{v_i\}$ satisfies the odd condition.
\end{lemma}

\begin{proof}
    Let $G'=G-V(B)\setminus\{v_1,v_2,v_3\}$ and let $\varphi'$ be an odd 5-coloring of $G'$.
    Without loss of generality, we may assume that $\varphi'(v_i)=i$ for each $i\in \{1,2,3\}$.
    Let $\varphi(v)=\varphi'(v)$ for every vertex $v\in V(G')$, and we define $\varphi(u)$ for each $u\in V(B)\setminus\{v_1,v_2,v_3\}$ as follows.

    We first consider the case where $B$ is an ear of $G$.
    By the symmetry of $B$, it suffices to show that $\varphi'$ can be extended to a proper 5-coloring $\varphi$ of $G$ such that every vertex of $V(G)\setminus\{v_2,v_3\}$ satisfies the odd condition.
    We choose $\varphi(u_0)\in \{4,5\}$ so that $v_1$ satisfies the odd condition, and then it is easy to verify that every vertex of $V(G)\setminus\{v_2,v_3\}$ satisfies the odd condition.

    Next, we consider the case where $B$ is a one-hat of $G$.
    Let $V(B)=\{v_1,v_2,v_3,u_0,u_1\}$ such that $N_G(u_0)=\{v_1,v_2,v_3,u_1\}$ and $N_G(u_1)=\{v_1,v_2,u_0\}$.
    By the symmetry of $v_1$ and $v_2$, it suffices to show that the statement holds for the cases $i=1$ or $i=3$.
    We extend $\varphi'$ to a coloring $\varphi_1$ of $G$ by choosing $\varphi_1(u_0)\in\{4,5\}$ so that $v_3$ satisfies the odd condition and choosing $\varphi_1(u_1)\in\{3,4,5\}\setminus\{\varphi_1(u_0)\}$ so that $v_2$ satisfies the odd condition.
    We extend $\varphi'$ to a coloring $\varphi_3$ of $G$ by letting $(\varphi_3(u_0),\varphi_3(u_1))\in\{(4,1),(4,5),(5,1)\}$ so that both $v_1$ and $v_2$ satisfy the odd condition.
    In each case, it is easy to verify that $\varphi_i$ is a proper 5-coloring of $G$ such that every vertex other than $v_i$ satisfies the odd condition.

    Finally, we consider the case where $B$ is a one-hat plus of $G$.
    Let $V(B)=\{v_1,v_2,v_3\}\cup\{u_i\mid 0\leq i\leq 3\}$ such that $N_G(u_0)=\{v_1,v_2,v_3,u_1,u_2,u_3\}$, $N_G(u_1)=\{v_1,v_2,u_0,u_3\}$, $N_G(u_2)=\{v_2,v_3,u_0\}$ and $N_G(u_3)=\{v_1,u_0,u_1\}$.
    We let $\varphi_1(u_0)=4$ and choose $\varphi_1(u_2)\in\{1,5\}$ so that $v_3$ satisfies the odd condition.
    By the argument in the last paragraph, we can choose $\varphi_1(u_1)$ and $\varphi_1(u_3)$ so that every vertex in $\{v_2,u_0,u_1,u_3\}$ satisfies the odd condition.
    It is easy to verify that every vertex of $V(G)\setminus\{v_1\}$ satisfies the odd condition with respect to $\varphi_1$.
    Similarly, we can extend $\varphi'$ to a coloring $\varphi_2$ such that every vertex of $V(G)\setminus\{v_2\}$ satisfies the odd condition.
    We extend $\varphi'$ to a coloring $\varphi_3$ of $G$ by letting $\varphi_3(u_0)=4$, $\varphi_3(u_2)=5$, and choosing $\varphi_3(u_1)$ and $\varphi_3(u_3)$ so that every vertex of $\{v_1,v_2,u_1,u_3\}$ satisfies the odd condition.
    Since $d_G(u_0)=6$ and distinct four colors appear in the neighbors of $u_0$, $u_0$ satisfies the odd condition with respect to $\varphi_3$.
    Thus, every vertex of $V(G)\setminus\{v_3\}$ satisfies the odd condition with respect to $\varphi_3$.
\end{proof}

\subsection{Proof of Theorem~\ref{thm:3tree}}

\setcounter{case}{0}

Let $G$ be a 3-tree of order $n$.
The proof goes by induction on $n$.
If $n\leq 5$, then trivially $G$ is odd $5$-colorable.
We assume that $n\geq 6$, and that every 3-tree of order less than $n$ is odd $5$-colorable.
We first consider the case where $G$ has a one-hat or a one-hat plus with a vertex of specified degree in its root.

\begin{case}\label{case:3tree good hat}
  Suppose that $G$ contains a branch $B=B(\{y_1,y_2,y_3\},x)$ which is either a one-hat or a one-hat plus of $G$.
  If $d_G(y_1)$ is equal to $4$ or an odd integer, then $G$ is odd $5$-colorable.
\end{case}

\begin{proof}
  Let $G'=G-V(B)\setminus\{y_1,y_2,y_3\}$.
  By the induction hypothesis, $G'$ admits an odd 5-coloring $\varphi'$.
  By using Lemma~\ref{lem:3tree nearodd} to $B$, we can extend $\varphi'$ to a proper 5-coloring of $G$ such that every vertex other than $y_1$ satisfies the odd condition.
  By Observations~\ref{ob:odd deg vtx} and \ref{ob:small deg vtx}, $y_1$ satisfies the odd condition, and hence we extend $\varphi'$ to an odd 5-coloring of $G$.
\end{proof}

By Lemma~\ref{lem:3tree unavoidable}, $G$ has a subgraph $H\in\mathcal{H}^{(3)}(G)\cup\mathcal{T}^{(3)}(G)$ with the root $\{v_1,v_2,v_3\}$ that is neither a one-hat nor a one-hat plus of $G$.
Note that $H\notin \mathcal{H}^{(3)}_{1,0,0}(G)$ since $H$ is not a one-hat of $G$.
Let $G'=G-(V(H)\setminus\{v_1,v_2,v_3\})$.
By the induction hypothesis, $G'$ admits an odd 5-coloring $\varphi$.
Without loss of generality, we may assume that $\varphi(v_i)=i$ for each $i\in \{1,2,3\}$.

\begin{case}\label{case:3tree h001}
  $H\in\mathcal{H}^{(3)}_{a,b,c}(G)$ for some tuple $(a,b,c)$ with $1\leq a+b+c\leq 3$ and $c\geq 1$.
\end{case}

\begin{proof}
  Let $H=B(\{v_1,v_2,v_3\},u_0)\in \mathcal{H}^{(3)}_{a,b,c}(G)$ and suppose that there is a vertex $u_1\in V(H)$ such that $B_1=B(\{v_1,v_2,u_0\},u_1)$ is a one-hat plus of $G$.
  Let $B_2=B(\{v_2,v_3,u_0\},x)$ if there is a vertex $x\in (N_G(v_2)\cap N_G(v_3)\cap N_G(u_0))\setminus \{v_1\}$ and let $B_2$ be a triangle with three vertices $\{v_2,v_3,u_0\}$ otherwise.
  Similarly, let $B_3=B(\{v_3,v_1,u_0\},x)$ if there is a vertex $x\in (N_G(v_3)\cap N_G(v_1)\cap N_G(u_0))\setminus \{v_2\}$ and let $B_3$ be a triangle with three vertices $\{v_3,v_1,u_0\}$ otherwise.
  By the assumption that $H\in \mathcal{H}^{(3)}_{a,b,c}(G)$, each of $B_2$, $B_3$ is either an ear, a one-hat, a one-hat plus or a triangle of $G$.
  We first show the following subclaim.

  \begin{subcase}\label{subcase:3tree}
    $d_{B_1}(u_0)=4$.
  \end{subcase}
  \begin{proof}
    By the symmetry of $v_1$ and $v_2$, we may assume that $V(B_1)=\{v_1,v_2,u_0,u_1,u_2,u_3,u_4\}$ such that $N_G(u_1)=\{v_1,v_2,u_0,u_2,u_3,u_4\}$, $N_G(u_2)=\{v_1,v_2,u_1,u_4\}$, $N_G(u_3)=\{v_1,u_0,u_1\}$, and $N_G(u_4)=\{v_2,u_1,u_2\}$.
    If both $B_2$ and $B_3$ are triangles of $G$, then we have $d_G(u_0)=5$, which falls into Case~\ref{case:3tree good hat}.
    Hence, we assume that at least one of $B_2$ and $B_3$ is not a triangle of $G$.

    Let $\varphi(u_0)=4$, $\varphi(u_1)=5$ and $\varphi(u_3)=3$.
    Suppose that $d_G(u_0)\leq 7$.
    Then, by using Lemma~\ref{lem:3tree nearodd}, we can choose colors for the vertices in $V(H)\setminus (V(B_1)\cup\{v_3\})$ so that every vertex of $V(H)\setminus V(B_1)$ satisfies the odd condition.
    Again, by Lemma~\ref{lem:3tree nearodd}, we choose $\varphi(u_2)$ and $\varphi(u_4)$ so that every vertex of $\{v_1,v_2,u_2,u_4\}$ satisfies the odd condition.
    By the choice of colors, it is easy to verify that every vertex other than $u_0$ and $u_1$ satisfies the odd condition.
    Furthermore, we have $d_G(u_0), d_G(u_1)\leq 7$ and four distinct colors appear in the neighborhood of each of $u_0$ and $u_1$.
    Thus, $\varphi$ is an odd 5-coloring of $G$.
    
    We assume that $d_G(u_0)\geq 8$, which implies that at least one of $B_1$ and $B_2$ is a one-hat or a one-hat plus of $G$. 
    By using Lemma~\ref{lem:3tree nearodd} to $B_2$ and $B_3$, we can choose colors for the vertices in $V(H)\setminus (V(B_1)\cup\{v_3\})$ so that every vertex of $V(H)\setminus (V(B_1)\setminus\{u_0\})$ satisfies the odd condition.
    By Lemma~\ref{lem:3tree nearodd}, we can choose $\varphi(u_2)$ and $\varphi(u_4)$ so that every vertex of $\{v_1,v_2,u_2,u_4\}$ satisfies the odd condition, and the resulting coloring $\varphi$ is an odd 5-coloring of $G$.
  \end{proof}

  Now we return to the proof of Case~\ref{case:3tree h001}.
  Let $V(B_1)=\{v_1,v_2,u_0,u_1,u_2,u_3,u_4\}$.
  By Subcase~\ref{subcase:3tree} and the symmetry of $v_1$ and $v_2$, we may assume one of the followings:
  \begin{enumerate}[label=(\alph*)]
    \item $N_G(u_1)=\{v_1,v_2,u_0,u_2,u_3,u_4\}$, $N_G(u_2)=\{v_1,u_0,u_1,u_4\}$, $N_G(u_3)=\{v_1,v_2,u_1\}$, and $N_G(u_4)=\{u_0,u_1,u_2\}$ (Figure \ref{fig:odd 3tree h001-1}), or 
    \item $N_G(u_1)=\{v_1,v_2,u_0,u_2,u_3,u_4\}$, $N_G(u_2)=\{v_1,u_0,u_1,u_4\}$, $N_G(u_3)=\{v_2,u_0,u_1\}$, and $N_G(u_4)=\{v_1,u_1,u_2\}$ (Figure \ref{fig:odd 3tree h001-2}).
  \end{enumerate}

  \begin{figure}
    \centering
    \begin{minipage}{0.45\columnwidth}
      \centering
      \begin{tikzpicture}[roundnode/.style={circle, draw=black,fill=white, minimum size=1.5mm, inner sep=0pt}]
        \node [roundnode] (v1) at (90:2.4){};
        \node [roundnode] (v2) at (210:2.4){};
        \node [roundnode] (v3) at (330:2.4){};
        \node [roundnode] (u0) at (330:0.8){};
        \node [roundnode] (u1) at (150:0.6){};
        \node [roundnode] (u2) at (90:1){};
        \node [roundnode] (u3) at (150:0.9){};
        \node [roundnode] (u4) at (90:0.3){};
  
        \foreach \x in {1,2,3} \draw (u0)--(v\x);
        \draw (v1)--(v2)--(v3)--(v1);
        \foreach \x in {1,2} \draw (u1)--(v\x);
        \foreach \x in {0,2,3,4} \draw (u1)--(u\x);
        \foreach \x in {1,2} \draw (u3)--(v\x);
        \foreach \x in {0,1,2} \draw (u4)--(u\x);
        \draw (u0)--(u2)--(v1);
        \draw (v1)--(90:2.7);
        \draw (v2)--(210:2.7);
        \draw (v3)--(330:2.7);
  
        \node at (97:2.4){$v_1$};
        \node at (217:2.4){$v_2$};
        \node at (323:2.4){$v_3$};
        \node at (305:1.1){$B_2$};
        \node at (355:1.1){$B_3$};
        \node at (300:0.4){$u_0$};
      \end{tikzpicture}
      \subcaption{}
      \label{fig:odd 3tree h001-1}
    \end{minipage}
    \begin{minipage}{0.45\columnwidth}
      \centering
      \begin{tikzpicture}[roundnode/.style={circle, draw=black,fill=white, minimum size=1.5mm, inner sep=0pt}]
        \node [roundnode] (v1) at (90:2.4){};
        \node [roundnode] (v2) at (210:2.4){};
        \node [roundnode] (v3) at (330:2.4){};
        \node [roundnode] (u0) at (330:0.8){};
        \node [roundnode] (u1) at (150:0.6){};
        \node [roundnode] (u2) at (30:0.3){};
        \node [roundnode] (u3) at (210:0.5){};
        \node [roundnode] (u4) at (100:0.7){};
  
        \foreach \x in {1,2,3} \draw (u0)--(v\x);
        \draw (v1)--(v2)--(v3)--(v1);
        \foreach \x in {1,2} \draw (u1)--(v\x);
        \foreach \x in {0,2,3,4} \draw (u1)--(u\x);
        \draw (v2)--(u3)--(u0);
        \draw (u0)--(u2)--(u4)--(v1)--(u2);
        \draw (v1)--(90:2.7);
        \draw (v2)--(210:2.7);
        \draw (v3)--(330:2.7);
        \draw[->] (30:1.5).. controls (30:1) ..(340:0.8);
  
        \node at (97:2.4){$v_1$};
        \node at (217:2.4){$v_2$};
        \node at (323:2.4){$v_3$};
        \node at (305:1.1){$B_2$};
        \node at (355:1.1){$B_3$};
        \node at (30:1.8){$u_0$};
      \end{tikzpicture}
      \subcaption{}
      \label{fig:odd 3tree h001-2}
    \end{minipage}
    \caption{Graphs in Case~\ref{case:3tree h001}}
  \end{figure}
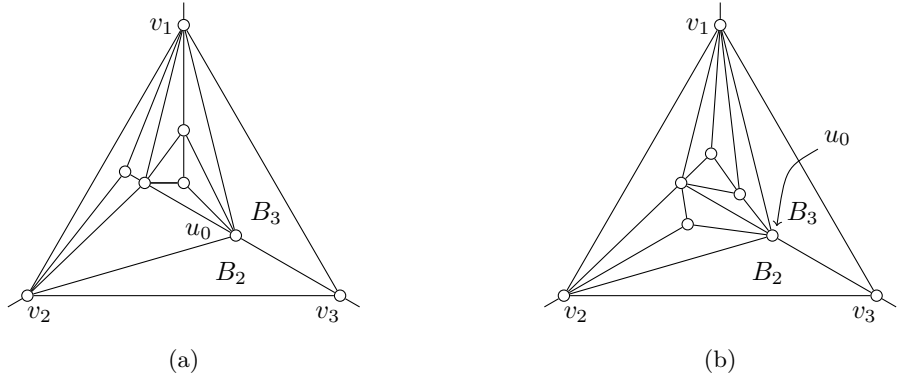

  We give a strategy for a coloring of $H$ that works for both of these two cases as follows.
  Suppose first that both $B_2$ and $B_3$ are triangles of $G$.
  We let $(\varphi(u_0),\varphi(u_1))$ be $(4,5)$ or $(5,4)$ so that $v_3$ satisfies the odd condition and choose $\varphi(u_3)$ so that $v_2$ satisfies the odd condition.
  By using Lemma~\ref{lem:3tree nearodd} to $B(\{v_1,u_0,u_1\},u_2)$, we choose $\varphi(u_2)$ and $\varphi(u_4)$ so that every vertex of $\{v_1,u_1,u_2,u_4\}$ satisfies the odd condition.
  It is easy to verify that every vertex other than $u_0$ satisfies the odd condition, and $u_0$ satisfies the odd condition since $d_G(u_0)=6$ and four colors $\{1,2,3,5\}$ appear in the neighborhood.
  Thus, $\varphi$ is an odd 5-coloring of $G$.
  
  Next, we assume that neither $B_2$ nor $B_3$ is a triangle of $G$.
  In this case, we let $\varphi(u_0)=4$, $\varphi(u_1)=5$, and $\varphi(u_3)=3$.
  By using Lemma~\ref{lem:3tree nearodd} to $B_2$, we can color $V(B_2)\setminus\{v_2,v_3,u_0\}$ so that every vertex of $V(B_2)\setminus\{v_3,u_0\}$ satisfies the odd condition.
  Then again by using Lemma~\ref{lem:3tree nearodd} to $B_3$, we can color $B_3$ so that every vertex of $V(B_3)\setminus\{v_1,u_0\}$ satisfies the odd condition.
  Finally, we color $B(\{v_1,u_0,u_1\},u_2)$ so that every vertex of $\{v_1,u_0,u_2,u_4\}$ satisfies the odd condition.
  It is easy to verify that every vertex other than $u_1$ satisfies the odd condition, and $u_1$ satisfies the odd condition since $d_G(u_1)=6$ and four colors $\{1,2,3,4\}$ appear in the neighborhood.
  Thus, $\varphi$ is an odd 5-coloring of $G$.

  Hence, we may assume that one of $B_2$ and $B_3$ is a triangle and the other is not.
  Then we have $d_G(u_0)\in \{7,8,9\}$, and we are done for the cases $d_G(u_0)\in \{7,9\}$ by Case~\ref{case:3tree good hat}.
  Thus, $d_G(u_0)=8$, which implies that either $d_{B_2}(u_0)=4$ and $B_3$ is a triangle, or $d_{B_3}(u_0)=4$ and $B_2$ is a triangle.
  If $d_{B_i}(u_0)=4$ and $B_i$ is a one-hat plus of $G$ for some $i\in\{2,3\}$, then it falls into Subclaim~\ref{subcase:3tree}.
  Hence, we may assume that one of $B_2$ and $B_3$ is a one-hat of $G$, and the other is a triangle.
  For each case, we extend $\varphi$ to $G$ as follows:
  \begin{itemize}
    \item When $B_2$ is a one-hat of $G$, we let $\varphi(u_0)=4$, $\varphi(u_1)=5$ and $\varphi(u_3)=3$.
    We use Lemma~\ref{lem:3tree nearodd} to $B_2$ to color $V(B_2)\setminus\{v_2,v_3,u_0\}$ so that every vertex of $V(B_2)\setminus\{u_0\}$ satisfies the odd condition and use Lemma~\ref{lem:3tree nearodd} to $B(\{v_1,u_0,u_1\},u_2)$ to choose $\varphi(u_2)$ and $\varphi(u_4)$ so that every vertex of $\{v_1,u_0,u_2,u_4\}$ satisfies the odd condition.
    \item When $B_3$ is a one-hat of $G$, we let $\varphi(u_0)=4$, $\varphi(u_1)=5$, $\varphi(u_2)=3$, and choose $\varphi(u_3)$ so that $v_2$ satisfies the odd condition.
    Let $x\in\{v_1,u_0\}$ so that $N_G(u_4)=\{u_0,u_1,x\}$.
    We use Lemma~\ref{lem:3tree nearodd} to $B_3$ to color $V(B_3)\setminus\{v_3,v_1,u_0\}$ so that every vertex of $V(B_3)\setminus\{x\}$ satisfies the odd condition, and choose $\varphi(u_4)\in\{1,2,4\}\setminus\{\varphi(x)\}$ so that $x$ satisfies the odd condition.
  \end{itemize}
  In both cases, every vertex other than $u_1$ satisfies the odd condition by the choice of colors, and $u_1$ satisfies the odd condition since $d_G(u_1)=6$ and four colors $\{1,2,3,4\}$ appear in the neighborhood.
  Thus, $\varphi$ is an odd 5-coloring of $G$.
  This completes the proof of Case~\ref{case:3tree h001}.
\end{proof}

\begin{case}\label{case:3tree h200}
  $H\in\mathcal{H}^{(3)}_{2,0,0}(G)$.
\end{case}

\begin{proof}
  Let $V(H)=\{v_1,v_2,v_3,u_0,u_1,u_2\}$ such that $N_G(u_0)=\{v_1,v_2,v_3,u_1,u_2\}$, $N_G(u_1)=\{v_1,v_2,u_0\}$, and $N_G(u_2)=\{v_2,v_3,u_0\}$.
  As $d_G(u_0)=5$, $u_0$ satisfies the odd condition with respect to any proper 5-coloring of $G$.
  If $|\varphi^{-1}(4)\cap N_{G'}(v_2)|$ is even, then we let $\varphi(u_0)=4$, choose $\varphi(u_1)\in \{3,5\}$ so that $v_1$ satisfies the odd condition, and choose $\varphi(u_2)\in \{1,5\}$ so that $v_3$ satisfies the odd condition.
  Every vertex other than $v_2$ satisfies the odd condition by the choice of colors, and $v_2$ satisfies the odd condition since $|\varphi^{-1}(4)\cap N_G(v_2)|=|\varphi^{-1}(4)\cap N_{G'}(v_2)|+1$ is odd.
  Thus, $\varphi$ is an odd 5-coloring of $G$.
  Hence, we may assume that $|\varphi^{-1}(4)\cap N_{G'}(v_2)|$ is odd.
  If $|\varphi^{-1}(4)\cap N_{G'}(v_1)|$ is even, then we let $\varphi(u_0)=4$, choose $\varphi(u_2)\in \{1,5\}$ so that $v_3$ satisfies the odd condition, and choose $\varphi(u_1)\in \{3,5\}$ so that $v_2$ satisfies the odd condition.
  Every vertex other than $v_1$ satisfies the odd condition by the choice of colors, and $v_1$ satisfies the odd condition since $|\varphi^{-1}(4)\cap N_G(v_1)|=|\varphi^{-1}(4)\cap N_{G'}(v_1)|+1$ is odd.
  Thus, $\varphi$ is an odd 5-coloring of $G$.
  Hence, we may assume that $|\varphi^{-1}(4)\cap N_{G'}(v_1)|$ is odd.
  By the symmetry of $v_1$ and $v_3$, we may also assume that $|\varphi^{-1}(4)\cap N_{G'}(v_3)|$ is odd.
  Then we let $\varphi(u_0)=5$, $\varphi(u_1)=3$ and $\varphi(u_2)=1$.
  Since $|\varphi^{-1}(4)\cap N_G(v_i)|=|\varphi^{-1}(4)\cap N_{G'}(v_i)|$ is odd for each $i\in \{1,2,3\}$, $\varphi$ is an odd 5-coloring of $G$.
\end{proof}

\begin{case}\label{case:3tree h300}
  $H\in\mathcal{H}^{(3)}_{3,0,0}(G)$.
\end{case}

\begin{proof}
  Let $V(H)=\{v_1,v_2,v_3\}\cup \{u_i\mid 0\leq i\leq 3\}$ such that $N_G(u_0)=\{v_1,v_2,v_3,u_1,u_2,u_3\}$, $N_G(u_1)=\{v_1,v_2,u_0\}$, $N_G(u_2)=\{v_2,v_3,u_0\}$, and $N_G(u_3)=\{v_3,v_1,u_0\}$.
  If $|\varphi^{-1}(4)\cap N_{G'}(v_1)|$ is even, then we let $\varphi(u_0)=4$, $\varphi(u_1)=5$, choose $\varphi(u_2)\in \{1,5\}$ so that $v_2$ satisfies the odd condition, and choose $\varphi(u_3)\in \{2,5\}$ so that $v_3$ satisfies the odd condition.
  Every vertex other than $\{v_1,u_0\}$ satisfies the odd condition by the choice of colors.
  The vertex $u_0$ satisfies the odd condition since $d_G(u_0)=6$ and four colors $\{1,2,3,4\}$ appear in the neighborhood, and $v_1$ satisfies the odd condition since $|\varphi^{-1}(4)\cap N_G(v_1)|=|\varphi^{-1}(4)\cap N_{G'}(v_1)|+1$ is odd.
  Thus, $\varphi$ is an odd 5-coloring of $G$.
  Hence, we may assume that $|\varphi^{-1}(4)\cap N_{G'}(v_1)|$ is odd.
  By the symmetry of $v_1$, $v_2$ and $v_3$, we may also assume that $|\varphi^{-1}(4)\cap N_{G'}(v_i)|$ for each $i\in\{2,3\}$.
  Then let $\varphi(u_0)=5$ and $\varphi(u_1)=\varphi(u_2)=\varphi(u_3)=4$.
  The vertex $u_0$ satisfies the odd condition since $|\varphi^{-1}(4)\cap N_G(u_0)|=3$.
  Furthermore, $|\varphi^{-1}(4)\cap N_G(v_i)|=|\varphi^{-1}(4)\cap N_{G'}(v_i)|+2$ is odd for each $i\in \{1,2,3\}$.
  Hence, $\varphi$ is an odd 5-coloring of $G$.
\end{proof}

\begin{case}\label{case:3tree h010}
  $H\in \mathcal{H}^{(3)}_{0,1,0}(G)$.
\end{case}

\begin{proof}
  Let $V(H)=\{v_1,v_2,v_3,u_0,u_1,u_2\}$ such that $B_1=B(\{v_1,v_2,u_0\},u_1)$ is a one-hat of $G$.
  Since $d_G(u_0)\in \{4,5\}$, we are done by Case~\ref{case:3tree good hat}.
\end{proof}

\begin{case}\label{case:3tree h110}
  $H\in \mathcal{H}^{(3)}_{1,1,0}(G)$ and $H$ is not a one-hat plus of $G$.
\end{case}

\begin{proof}
  Let $V(H)=\{v_1,v_2,v_3\}\cup\{u_i\mid 0\leq i\leq 3\}$ such that $B(\{v_1,v_2,u_0\},u_1)$ is a one-hat of $G$ and $B(\{v_2,v_3,u_0\},u_2)$ is an ear of $G$.
  Then we have $d_G(u_0)\in \{5,6\}$. 
  If $d_G(u_0)=5$, then we are done by Case~\ref{case:3tree good hat}.
  Hence, we assume that $d_G(u_0)=6$.
  Since $H$ is not a one-hat plus of $G$, we set $N_G(u_0)=\{v_1,v_2,v_3,u_1,u_2,u_3\}$, $N_G(u_1)=\{v_1,v_2,u_0,u_3\}$, $N_G(u_2)=\{v_2,v_3,u_0\}$ and $N_G(u_3)=\{v_2,u_0,u_1\}$ (Figure~\ref{fig:odd 3tree h110}).
  Let $\varphi(u_0)=4$.
  We choose $\varphi(u_2)\in \{1,5\}$ so that $v_3$ satisfies the odd condition, and we use Lemma~\ref{lem:3tree nearodd} to $B(\{v_1,v_2,u_0\},u_1)$ to color $\{u_1,u_3\}$ so that every vertex of $\{v_1,v_2,u_1,u_3\}$ satisfies the odd condition.
  Every vertex other than $u_0$ satisfies the odd condition by the choice of colors, and $u_0$ satisfies the odd condition since $\varphi^{-1}(2)\cap N_G(u_0)=\{v_2\}$.
  Thus, $\varphi$ is an odd 5-coloring of $G$.
\end{proof}

\begin{figure}
  \centering
    \begin{tikzpicture}[roundnode/.style={circle, draw=black,fill=white, minimum size=1.5mm, inner sep=0pt}]
      \node [roundnode] (v1) at (90:2.4){};
      \node [roundnode] (v2) at (210:2.4){};
      \node [roundnode] (v3) at (330:2.4){};
      \node [roundnode] (u0) at (0:0){};
      \node [roundnode] (u3) at (180:0.7){};
      \node [roundnode] (u2) at (270:0.6){};
      \node [roundnode] (u1) at (120:0.9){};

      \foreach \x in {1,2,3}{\draw (u0)--(v\x);
      \draw (u0)--(u\x);}
      \draw (v1)--(v2)--(v3)--(v1);
      \foreach \x in {1,2} \draw (u1)--(v\x);
      \foreach \x in {0,1} \draw (u3)--(u\x);
      \draw (u3)--(v2);
      \foreach \x in {2,3} \draw (u2)--(v\x);
      \draw (v1)--(90:2.7);
      \draw (v2)--(210:2.7);
      \draw (v3)--(330:2.7);

      \node at (97:2.4){$v_1$};
      \node at (217:2.4){$v_2$};
      \node at (323:2.4){$v_3$};
      \node at (30:0.3){$u_0$};
      \node at (103:0.9){$u_1$};
      \node at (152:0.45){$u_3$};
      \node at (270:0.9){$u_2$};
    \end{tikzpicture}
    \caption{A graph in Case~\ref{case:3tree h110}.}
    \label{fig:odd 3tree h110}
\end{figure}
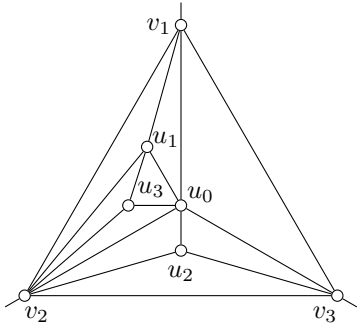

\begin{case}\label{case:3tree h210}
  $H\in \mathcal{H}^{(3)}_{2,1,0}(G)$.
\end{case}

\begin{proof}
  Let $V(H)=\{v_1,v_2,v_3\}\cup \{u_i\mid 0\leq i\leq 4\}$ such that $B_1=B(\{v_1,v_2,u_0\},u_1)$ is a one-hat of $G$ and both $B_2=B(\{v_2,v_3,u_0\},u_2)$ and $B_3=B(\{v_3,v_1,u_0\},u_3)$ are ears of $G$.
  We have $d_G(u_0)\in\{6,7\}$, and the case $d_G(u_1)=7$ falls into Case~\ref{case:3tree good hat}.
  Hence, we may assume that $d_G(u_0)=6$, which implies that $d_{B_1}(u_0)=3$.
  Then we set $N_G(u_0)=\{v_1,v_2,v_3,u_1,u_2,u_3\}$, $N_G(u_1)=\{v_1,v_2,u_0,u_4\}$, $N_G(u_2)=\{v_2,v_3,u_0\}$, $N_G(u_3)=\{v_3,v_1,u_0\}$, and $N_G(u_4)=\{v_1,v_2,u_1\}$.
  We let $\varphi(u_0)=4$, $\varphi(u_2)=5$, and choose $\varphi(u_3)\in \{2,5\}$ so that $v_3$ satisfies the odd condition.
  By using Lemma~\ref{lem:3tree nearodd} to $B_1$, we choose $\varphi(u_1)$ and $\varphi(u_4)$ so that every vertex of $\{v_1,v_2,u_1,u_4\}$ satisfies the odd condition.
  Every vertex other than $u_0$ satisfies the odd condition by the choice of colors, and $u_0$ satisfies the odd condition since $d_G(u_0)=6$ and four colors $\{1,2,3,5\}$ appear in the neighborhood.
  Thus, $\varphi$ is an odd 5-coloring of $G$.
\end{proof}

\begin{case}\label{case:3tree h020}
  $H\in \mathcal{H}^{(3)}_{0,2,0}(G)$.
\end{case}

\begin{proof}
  Let $V(H)=\{v_1,v_2,v_3\}\cup\{u_i\mid 0\leq i\leq 4\}$ such that $B_1=B(\{v_1,v_2,u_0\},u_1)$ and $B_2=B(\{v_2,v_3,u_0\},u_2)$ are both one-hats of $G$.
  We have $d_G(u_0)\in\{5,6,7\}$, and we are done for cases $d_G(u_0)\in \{5,7\}$ by Case~\ref{case:3tree good hat}.
  Suppose that $d_G(u_0)=6$. 
  Then by the symmetry of $v_1$ and $v_3$, we may assume that $V(B_1)=\{v_1,v_2,u_0,u_1,u_3\}$ such that $N_G(u_1)=\{v_1,v_2,u_0,u_3\}$ and $N_G(u_3)=\{v_1,v_2,u_1\}$.
  We let $\varphi(u_0)=4$, $\varphi(u_1)=5$, and choose $\varphi(u_3)\in \{3,4\}$ so that $v_1$ satisfies the odd condition.
  By using Lemma~\ref{lem:3tree nearodd} to $B_2$, we color $V(B_2)\setminus\{v_2,v_3,u_0\}$ so that every vertex of $V(B_2)\setminus\{u_0\}$ satisfies the odd condition.
  Every vertex other than $u_0$ satisfies the odd condition by the choice of colors, and $u_0$ satisfies the odd condition since $d_G(u_0)=6$ and four colors $\{1,2,3,5\}$ appear in the neighborhood.
  Thus, $\varphi$ is an odd 5-coloring of $G$.
\end{proof}

\begin{case}\label{case:3tree h120}
  $H\in \mathcal{H}^{(3)}_{1,2,0}(G)\cup \mathcal{H}^{(3)}_{0,3,0}(G)$.
\end{case}

\begin{proof}
  Let $V(H)=\{v_1,v_2,v_3\}\cup\{u_i\mid 0\leq i\leq t\}$ for some $t\in \{5,6\}$ such that both $B_1=B(\{v_1,v_2,u_0\},u_1)$ and $B_2=B(\{v_2,v_3,u_0\},u_2)$ are one-hats of $G$, and $B_3=B(\{v_3,v_1,u_0\},u_3)$ is either an ear or a one-hat of $G$.
  ($t=5$ when $B_3$ is an ear and $t=6$ when $B_3$ is a one-hat.)
  Let $V(B_1)=\{v_1,v_2,u_0,u_1,u_4\}$ and $V(B_2)=\{v_2,v_3,u_0,u_2,u_5\}$.
  Let $V(B_3)=\{v_3,v_1,u_0,u_3\}$ if $B_3$ is an ear, and let $V(B_3)=\{v_3,v_1,u_0,u_3,u_6\}$ if $B_3$ is a one-hat.
  We have $d_G(u_0)\in \{6,7,8\}$, and the case $d_G(u_0)=7$ is done in Case~\ref{case:3tree good hat}.

  Suppose first that $d_G(u_0)=6$. 
  Then we have $d_{B_3}(u_0)=3$.
  When $B_3$ is a one-hat, $N_G(u_6)=\{v_3,v_1,u_3\}$.
  Let $\varphi(u_0)=4$, $\varphi(u_3)=5$, and if $B_3$ is a one-hat, then let $\varphi(u_6)=2$.
  We use Lemma~\ref{lem:3tree nearodd} to $B_1$ to color $\{u_1,u_4\}$ so that every vertex of $V(B_1)\setminus\{u_0\}$ satisfies the odd condition, then we apply Lemma~\ref{lem:3tree nearodd} to $B_2$ to color $\{u_2,u_5\}$ so that every vertex of $V(B_2)\setminus\{u_0\}$ satisfies the odd condition.
  Every vertex other than $u_0$ satisfies the odd condition by the choice of colors, and $u_0$ satisfies the odd condition since $d_G(u_0)=6$ and four colors $\{1,2,3,5\}$ appear in the neighborhood.
  Thus, $\varphi$ is an odd 5-coloring of $G$.

  Suppose that $d_G(u_1)=8$.
  We may assume that $d_{B_1}(u_0)=d_{B_2}(u_0)=4$ and $d_{B_3}(u_0)=3$;
  If $B_3$ is an ear, then we have $d_{B_3}(u_0)=3$ and thus both $d_{B_1}(u_0)$ and $d_{B_2}(u_0)$ must be equal to $4$.
  Otherwise, two of $\{d_{B_1}(u_0), d_{B_2}(u_0), d_{B_3}(u_0)\}$ are equal to $4$ and the other is equal to $3$. 
  By the symmetry of $B_1$, $B_2$, $B_3$, we may assume that $d_{B_3}(u_0)=3$.
  Note that every vertex of $V(H)\setminus\{u_0, u_6\}$ is adjacent to $u_0$, and thus a color assigned to $u_0$ appears exactly once in $N_G(v_2)\cap V(H)$ for any proper coloring of $G$.

  We suppose that $|\varphi^{-1}(4)\cap N_{G'}(v_1)|$ is even.
  Then let $\varphi(u_0)=4$, $\varphi(u_3)=5$, and if $B_3$ is a one-hat, then let $\varphi(u_6)=2$.
  We use Lemma~\ref{lem:3tree nearodd} to $B_2$ to color $\{u_2,u_5\}$ so that every vertex of $\{v_2,v_3,u_2,u_5\}$ satisfies the odd condition, 
  and we use Lemma~\ref{lem:3tree nearodd} to $B_1$ to color $\{u_1,u_4\}$ so that every vertex of $\{v_2,u_0,u_1,u_4\}$ satisfies the odd condition.
  Every vertex other than $v_1$ satisfies the odd condition by the choice of colors, and $v_1$ satisfies the odd condition since $|\varphi^{-1}(4)\cap N_G(v_1)|=|\varphi^{-1}(4)\cap N_{G'}(v_1)|+1$ is odd.
  Thus, $\varphi$ is an odd 5-coloring of $G$.
  Hence, we may assume that $|\varphi^{-1}(4)\cap N_{G'}(v_1)|$ is odd.
  By the symmetry of $v_1$ and $v_3$, we may also assume that $|\varphi^{-1}(4)\cap N_{G'}(v_3)|$ is odd.
  
  Next, we suppose that $|\varphi^{-1}(4)\cap N_{G'}(v_2)|$ is even.
  Then let $\varphi(u_0)=4$, $\varphi(u_3)=5$, and if $B_3$ is a one-hat, then let $\varphi(u_6)=2$.
  We use Lemma~\ref{lem:3tree nearodd} to $B_1$ to color $\{u_1,u_4\}$ so that every vertex of $\{v_1,u_0,u_1,u_4\}$ satisfies the odd condition, and we use Lemma~\ref{lem:3tree nearodd} to $B_2$ to color $\{u_2,u_5\}$ so that every vertex of $\{v_3,u_0,u_2,u_5\}$ satisfies the odd condition.
  Every vertex other than $v_2$ satisfies the odd condition by the choice of colors, and $v_2$ satisfies the odd condition since $|\varphi^{-1}(4)\cap N_G(v_2)|=|\varphi^{-1}(4)\cap N_{G'}(v_2)|+1$ is odd.
  Thus, $\varphi$ is an odd 5-coloring of $G$.
  Hence, we may assume that $|\varphi^{-1}(4)\cap N_{G'}(v_2)|$ is odd.

  Now we let $\varphi(u_0)=5$ and $\varphi(u_1)=\varphi(u_2)=\varphi(u_3)=4$.
  We choose $\varphi(u_4)$, $\varphi(u_5)$, and $\varphi(u_6)$ arbitrarily so that the resulting coloring of $G$ is proper.
  For each $i\in \{1,2,3\}$, $|\varphi^{-1}(4)\cap N_G(v_i)|=|\varphi^{-1}(4)\cap N_{G'}(v_i)|+2$ is odd and thus $v_i$ satisfies the odd condition.
  Furthermore, since $|\varphi^{-1}(4)\cap N_G(u_0)|=3$, $u_0$ satisfies the odd conditions.
  Thus, $\varphi$ is an odd 5-coloring of $G$.
\end{proof}

\begin{case}\label{case:3tree t001}
  $H\in\mathcal{T}^{(3)}_{a,b,c}(G)$ for some tuple $(a,b,c)$ with $a+b+c=2$ and $c\geq 1$.
\end{case}

\begin{proof}
  Let $V(H)=V(B_u)\cup V(B_w)$ where $B_u=B(\{v_1,v_2,v_3\},u_0)$ is a one-hat plus of $G$ and $B_w=B(\{v_1,v_2,v_3\},w_0)$ is either an ear, a one-hat, or a one-hat plus of $G$.
  By the symmetry of $v_1$, $v_2$, and $v_3$, we may assume that $V(B_u)=\{v_1,v_2,v_3\}\cup\{u_i\mid 0\leq i\leq 3\}$ such that $N_G(u_0)=\{v_1,v_2,v_3,u_1,u_2,u_3\}$, $N_G(u_1)=\{v_1,v_2,u_0,u_3\}$, $N_G(u_2)=\{v_2,v_3,u_0\}$, and $N_G(u_3)=\{v_1,u_0,u_1\}$.
  We let $\varphi(u_0)=4$ and $\varphi(u_2)=5$.
  By using Lemma~\ref{lem:3tree nearodd} to $B_w$, we choose colors for $V(B_w)\setminus\{v_1,v_2,v_3\}$ so that every vertex of $V(B_w)\setminus\{v_1,v_2\}$ satisfies the odd condition.
  Then we use Lemma~\ref{lem:3tree nearodd} to a one-hat $B(\{v_1,v_2,u_0\},u_1)$ to color $u_1$ and $u_3$ so that every vertex of $\{v_1,v_2,u_1,u_3\}$ satisfies the odd condition.
  Every vertex other than $u_0$ satisfies the odd condition by the choice of colors, and $u_0$ satisfies the odd condition since $d_G(u_0)=6$ and four colors $\{1,2,3,5\}$ appear in the neighborhood.
  Thus, $\varphi$ is an odd 5-coloring of $G$.
\end{proof}

\begin{case}\label{case:3tree t200}
  $H\in \mathcal{T}^{(3)}_{2,0,0}(G)$.
\end{case}

\begin{proof}
  Let $V(H)=\{v_1,v_2,v_3,u_0,w_0\}$ such that $N_G(u_0)=N_G(w_0)=\{v_1,v_2,v_3\}$.
  We let $\varphi(u_0)=\varphi(w_0)=4$.
  Then, for every vertex $x$ of $G'$ and every color $i\in \{1,2,3,4,5\}$, $|\varphi^{-1}(i)\cap N_G(x)|$ and $|\varphi^{-1}(i)\cap N_{G'}(x)|$ have the same parity, and hence $\varphi$ is an odd 5-coloring of $G$.
\end{proof}

\begin{case}\label{case:3tree t110}
  $H\in \mathcal{T}^{(3)}_{1,1,0}(G)$. 
\end{case}

\begin{proof}
  Let $V(H)=V(B_u)\cup V(B_w)$ where $B_u=B(\{v_1,v_2,v_3\},u_0)$ is a one-hat of $G$ and $B_w=B(\{v_1,v_2,v_3\},w_0)$ is an ear of $G$.
  By the symmetry of $v_1$, $v_2$ and $v_3$, we may assume that $V(B_u)=\{v_1,v_2,v_3,u_0,u_1\}$ such that $N_G(u_0)=\{v_1,v_2,v_3,u_1\}$ and $N_G(u_1)=\{v_1,v_2,u_0\}$.
  We first choose $\varphi(u_1)\in \{3,4,5\}$ so that $|\varphi^{-1}(i_1)\cap (N_G(v_1)\setminus\{u_0,w_0\})|$ is odd for some color $i_1\in [5]$ and $|\varphi^{-1}(i_2)\cap (N_G(v_2)\setminus\{u_0,w_0\})|$ is odd for some color $i_2\in [5]$.
  We let $\varphi(u_0)=\varphi(w_0)=j$ where $j\in\{4,5\}\setminus\{\varphi(u_1)\}$ to obtain a coloring of $G$.
  For $i\in \{1,2\}$, the vertex $v_i$ satisfies the odd condition since $|\varphi^{-1}(i_1)\cap N_G(v_1)|$ is odd.
  Similarly, $v_3$ satisfies the odd condition since $|\varphi^{-1}(i)\cap N_G(v_3)|$ and $|\varphi^{-1}(i)\cap N_{G'}(v_3)|$ have the same parity for every color $i\in \{1,2,3,4,5\}$.
  Thus, $\varphi$ is an odd 5-coloring of $G$.
\end{proof}

\begin{case}\label{case:3tree t020}
  $H\in \mathcal{T}^{(3)}_{0,2,0}(G)$. 
\end{case}

\begin{proof}
  Let $V(H)=V(B_u)\cup V(B_w)$ where both $B_u=B(\{v_1,v_2,v_3\},u_0)$ and $B_w=B(\{v_1,v_2,v_3\},w_0)$ are one-hats of $G$.
  Let $V(B_u)=\{v_1,v_2,v_3,u_0,u_1\}$ and $V(B_w)=\{v_1,v_2,v_3,w_0,w_1\}$.
  By the symmetry of $v_1$, $v_2$ and $v_3$, we may assume that
  \begin{enumerate}[label=(\alph*)]
    \item\label{case:odd 3tree t020 1} $N_G(u_0)=\{v_1,v_2,v_3,u_1\}$, $N_G(u_1)=\{v_1,v_2,u_0\}$, $N_G(w_0)=\{v_1,v_2,v_3,w_1\}$, and $N_G(w_1)=\{v_1,v_2,w_0\}$, or
    \item\label{case:odd 3tree t020 2} $N_G(u_0)=\{v_1,v_2,v_3,u_1\}$, $N_G(u_1)=\{v_1,v_2,u_0\}$, $N_G(w_0)=\{v_1,v_2,v_3,w_1\}$, and $N_G(w_1)=\{v_2,v_3,w_0\}$.
  \end{enumerate}
  If $H$ satisfies \ref{case:odd 3tree t020 1}, then we let $\varphi(u_0)=\varphi(w_0)=4$ and $\varphi(u_1)=\varphi(w_1)=5$.
  Then, for every vertex $x$ of $G'$ and every color $i\in \{1,2,3,4,5\}$, $|\varphi^{-1}(i)\cap N_G(x)|$ and $|\varphi^{-1}(i)\cap N_{G'}(x)|$ have the same parity, which implies that $\varphi$ is an odd 5-coloring of $G$.
  Hence, we may assume that $G$ satisfies $H$ satisfies \ref{case:odd 3tree t020 2}.
  Let $\varphi(u_0)=4$ and choose $\varphi(w_0)\in\{4,5\}$ so that $|\varphi^{-1}(4)\cap (N_G(v_1)\setminus\{u_1\})|$ is odd.
  We choose $\varphi(w_1)\in\{1,4,5\}\setminus\{\varphi(w_0)\}$ so that $v_3$ satisfies the odd condition, and then choose $\varphi(u_1)\in \{3,5\}$ so that $v_2$ satisfies the odd condition.
  By the choice of colors, every vertex other than $v_1$ satisfies the odd condition.
  Since $\varphi(u_1)\neq 4$, $|\varphi^{-1}(4)\cap N_G(v_1)|=|\varphi^{-1}(4)\cap (N_G(v_1)\setminus\{u_1\})|$ is odd and $v_1$ satisfies the odd condition.
  Thus, $\varphi$ is an odd 5-coloring of $G$.
  This completes the proof of Theorem~\ref{thm:3tree}.
\end{proof}

\end{document}